\documentclass[11pt,tbtags]{amsart}
\usepackage{amssymb}
\usepackage{amsmath}
\usepackage{amsthm}
\usepackage{dsfont}
\usepackage[swedish, english]{babel}
\usepackage[latin1]{inputenc}
\usepackage{psfrag}
\usepackage{graphicx}
\usepackage{geometry}
\usepackage{subfigure}
\usepackage{color}
\usepackage[numbers]{natbib}
\usepackage{mathtools}

\mathtoolsset{showonlyrefs}

\theoremstyle{plain}
\newcommand{\E}{\mathsf E}
\newcommand{\R}{\mathbb R}
\newcommand{\cC}{\mathcal C}
\newcommand{\cS}{\mathcal S}
\newcommand{\D}{\mathcal D}

\newcommand{\cA}{\mathcal A}

\newcommand{\cL}{\mathcal L}

\newcommand{\cF}{\mathcal F}
\newcommand{\bF}{\mathbb F}
\newcommand{\cT}{\mathcal T}

\newcommand{\ud}{\mathrm d}

\def\P{{\mathsf P}}

\def\Q{{\mathsf Q}}
\def\eps{{\varepsilon}}
\newcommand{\I}{\mathds{1}}

\newtheorem{theorem}{Theorem}[section]
\newtheorem{lemma}[theorem]{Lemma}
\newtheorem{corollary}[theorem]{Corollary}
\newtheorem{proposition}[theorem]{Proposition}

\newtheorem{assumption}[theorem]{Assumption}
\newtheorem{remark}[theorem]{Remark}
\theoremstyle{definition}

\title[A class of recursive optimal stopping problems]{A class of recursive optimal stopping problems \\ with applications to stock trading}
\author[Katia Colaneri]{Katia Colaneri}
\author[Tiziano De Angelis]{Tiziano De Angelis}
\subjclass[2000]{Primary 60G40; Secondary 91G80}
\keywords{optimal stopping theory; recursive optimal stopping problems; stock selling}
\address{K.~Colaneri, Department of Economics and Finance, University of Rome Tor Vergata, Via Columbia 2, 00133 Roma, Italy.}
\email{katia.colaneri@uniroma2.it}
\address{T.~De Angelis, Department ESOMAS, University of Turin, C.so Unione Sovietica, 218bis, 10134 Torino, Italy; Collegio Carlo Alberto, Piazza Arbarello 8, 10122, Torino, Italy;}
\email{tiziano.deangelis@unito.it}
\date{\today}
\thanks{{\em Acknowledgements} T.~De Angelis gratefully acknowledges support from EPSRC First Grant EP/R021201/1. This work was presented at the 2nd Leeds Conference on Stochastic Control, Ambiguity and Games (Leeds 2019). We thank E.~Bayraktar, J.~Ob\l\'oj and S.~Villeveuve for useful comments. Part of this work was written when both authors were affiliated to the School of Mathematics of the University of Leeds. The authors are grateful to two anonymous referees for insightful comments on the economic interpretation of the model.}

\begin{document}

\begin{abstract}
In this paper we introduce and solve a class of optimal stopping problems of recursive type. In particular, the stopping payoff depends directly on the value function of the problem itself. In a multi-dimensional Markovian setting we show that the problem is well posed, in the sense that the value is indeed the unique solution to a fixed point problem in a suitable space of continuous functions, and an optimal stopping time exists. We then apply our class of problems to a model for stock trading in two different market venues and we determine the optimal stopping rule in that case.
\end{abstract}

\maketitle

\section{Introduction}
In this paper we introduce a class of optimal stopping problems whose stopping payoff is defined in terms of the value function of the problem itself in a recursive way.
To gain some intuition on the nature of the problem we consider an individual who is allowed to choose the entry time to one of two possible (investment) projects, say $A$ and $B$, with random payoffs.
If the individual chooses project $A$, she immediately receives the corresponding payoff and the optimisation is over. Project $B$ has the potential for larger revenues but it is subject to a higher degree of uncertainty. In particular, when the individual chooses project $B$, she is not sure whether the project will succeed or not. This information is revealed only at a future (random) time. If project $B$ is successful, the individual receives the associated payoff. If instead the project fails, the optimisation must start afresh.

In our class of recursive optimal stopping problems, we consider a right-continuous, $\R^d$-valued strong Markov process $X$ and denote by $\E_x$ the expectation conditional upon $X_0=x$. We want to find a continuous function $v$ such that for every $x \in \R^d$
\begin{align}\label{eq:v}
v(x)=\!\sup_{(\tau, \alpha)\in \mathcal{D}}\E_x\left[e^{-r \tau} \varphi(X_\tau) \I_{\{\alpha=0\}}\! +\! e^{-r (\tau+\vartheta) }   \left( p \psi(X_{\tau+\vartheta})\!+\! (1\!-\!p) v(X_{\tau+\vartheta})\right) \I_{\{\alpha=1\}} \right].
\end{align}
Here $\tau$ denotes the decision time and $\alpha$ indicates the project to be chosen (a precise definition of the set $\mathcal{D}$ of admissible controls is given in Section \ref{sec:setting}). Functions $\varphi$ and $\psi$ are real-valued, continuous, with $\psi\ge\varphi$, and represent the revenues associated to project $A$ (corresponding to $\alpha=0$) and to project $B$ (corresponding to $\alpha=1$), respectively. The random variable $\vartheta$ is the delay associated to the output of project $B$ (the riskier one) and $p \in [0,1]$ is the probability of a positive outcome for such project. Notice that at time $\tau+\vartheta$ the optimiser learns if project $B$ has succeeded and the optimization is over (in which case she receives $\psi(X_{\tau+\theta})$) or whether the project failed and the optimisation starts afresh.

In the paper we first show that the problem above is equivalent to a recursive problem of stopping only (see Lemma \ref{lem:value}). Using this fact we prove that \eqref{eq:v} admits a unique fixed point $v$ in a suitable Banach space of continuous functions. Furthermore there exists an optimal couple $(\tau_*,\alpha^*)$ and the value function $v$ fulfils suitable (super)martingale properties (see Theorem \ref{thm:prob1} and Lemma \ref{lem:value}). The set-up and methodology are sufficiently general to allow the analysis of problems with both finite-time and infinite-time horizon.

In order to investigate in more detail the structure of optimal strategies for this class of problems, in Section \ref{sec:appl} we consider an example where $X$ is a two-dimensional geometric Brownian motion and the functions $\varphi$ and $\psi$ are linear. As explained in more detail below, this example is motivated by the problem of selling a stock in two different trading venues: the lit market (where the sale is certain) and a dark pool (where the sale is uncertain but more profitable). In this case we show that the state space can be reduced to one dimension (Proposition \ref{prop:dim}) thus allowing for a more explicit description of the geometry of the stopping set (Theorem \ref{thm:final} and Corollary \ref{cor:final}). Further, we prove that the value function $v$ is continuously differentiable (in both state variables) so that {\em smooth-fit} holds at the boundary of the stopping set. The latter feature was never observed in this class of problems. Finally, we characterise the value function $v$ as the unique solution of a suitable free boundary problem (Proposition \ref{prop:freeb}).

We take advantage of the detailed analysis of Section \ref{sec:appl} to illustrate features captured by our model via a comparison with a corresponding trading problem without recursion (Section \ref{sec:comparison}). In line with the intuition, we observe that the ability to reiterate the decision process creates an incentive for the trader to opt for the riskier sale more frequently in the recursive model than in a non-recursive one. At the same time, the sale in the lit market is delayed in the recursive model, compared to the other one.

In summary, our contributions are the following: (i) we introduce a class of recursive optimal stopping problems, that appear to be new in the literature; these models describe a decision maker who can repeat her investment decisions in the event of negative outcomes; (ii) we prove existence of a solution to the problem in broad generality; (iii) we solve in detail the recursive problem in a two-factor model for trading in different venues; in this case, the geometry of the {\em selling} and the {\em waiting} regions can be determined and compared with the solution of a corresponding non-recursive problem.

\subsection{Connection with the existing literature}
To the best of our knowledge, the class of problems that we introduce and solve has never been addressed in the mathematical literature. Some loose links can be drawn to control problems featuring recursive utility, optimal multiple stopping problems, and some optimal stopping or impulse control problems with delay. However the settings are very different as briefly explained below.

The study of control problems with recursive (intertemporal) utilities was initiated with the work of \citet{epstein2013substitution} in a discrete time setting and was later extended to continuous time models (see, e.g.,  \citet{duffie1992stochastic}). Recursive preferences were introduced to model investors' impatience and found applications in problems of asset pricing and optimal portfolio/consumption decisions. It should be noted, however, that problems with recursive utility are conceptually different from our problem: in the former the recursive structure arises due to the preferences of the decision maker, whereas in the latter the recursion is purely induced by the (exogenous) probability of an unsuccessful transaction.

Problems of optimal multiple stopping were motivated by applications to swing options in the commodity market. A swing contract allows the buyer to exercise a sequence of $n$ American options with a fixed minimum time lag between two subsequent rights of exercise. A rigorous mathematical formulation was given in \citet{carmona2008optimal} and a more recent account of further developments in the area can be found in the introduction of \citet{de2017optimal}. In general, a swing contract with $n$ rights has a payoff that depends on the value of the same contract with $n-1$ rights. In this sense, there is a recursive formulation of the problem. However, the recursion is of a different type to the one that we introduce in \eqref{eq:v}, where the payoff depends directly on the value for the {\em same} problem. As an extreme case, a swing contract with infinitely many rights of exercise can be seen as a problem of optimal stopping whose payoff depends on the value function itself. We can formally reduce our problem \eqref{eq:v} to that setting by taking both indicator variables equal to one, $p=0$ and by fixing a deterministic $\vartheta>0$.

Our set-up may be reminiscent of some impulse control problems but, in contrast to that class of problems, here the decision maker cannot influence the state dynamics (i.e., there is no control/impulse acting on $X$). For example, \citet{bayraktar2007effects} study a particular class of impulse control problems with delay that they rewrite with a recursive formulation. In their work an impulse exerted at a (stopping) time $\tau$ shifts the state dynamic to a different point in the state space after a (deterministic) delay $\Delta$.
In \cite{bayraktar2007effects} authors consider a one dimensional controlled diffusion and restrict the class of admissible strategies to so-called $(s,S)$-policies (i.e., upon hitting a level $s$ the process is shifted downwards to a new level $S<s$, after the delay $\Delta$). These assumptions allow them to adopt direct solution methods, often used in stopping problems for one-dimensional diffusions and based on a characterisation of the value via excessive functions (see, e.g., \citet{dynkin1969markov} and more recently \citet{dayanik2003optimal}).
Inventory problems with delivery lags and pending orders can also be cast as impulse control problems with delay and such problems were studied, for example, by \citet{bar1995explicit}. They show that the value function is the unique viscosity solution of a quasi-variational inequality in which the obstacle condition features a non-local term that depends on the delivery lag. They also find an explicit solution in a particular example with $(s,S)$-policy. The techniques employed in \citep{bayraktar2007effects} cannot be generalised to our multi-dimensional setting. Moreover, since in our case the underlying state process $X$ is uncontrolled, its position at time $\tau+\vartheta$ cannot be determined by actions of the decision maker. Hence, ideas used in \cite{bayraktar2007effects} and \cite{bar1995explicit} concerning $(s,S)$-policies do not apply to our setting.

Optimal stopping with delay (but without recursive structure) have been used in economics to model irreversible investment problems with time lag (see, e.g., \citet{bar1996investment}). In the mathematical literature, for example, \citet{oksendal2005optimal} considers deterministic delay and \citet{lempa2012optimal} considers random delay. In those settings the stopper can only choose a single stopping time and it is shown in \cite{oksendal2005optimal} and \cite{lempa2012optimal} that the problem reduces to a standard optimal stopping one. That corresponds to $\varphi\equiv 0$ and $p=1$ in our problem formulation, so that $\alpha=1$ is necessarily optimal and the recursive term vanishes.

\subsection{Motivation and examples}\label{sec:examples}
The key point in our model, which departs from the existing literature, is that an unsuccessful outcome of a risky project does not result in a direct cost for the decision maker and neither does it preclude the opportunity to try again\footnote{Imagine a house buyer who puts an offer on a property: If the offer is below the ask price it might be rejected, but will not prevent the buyer from bidding again and it carries no direct costs. If, on the contrary, the house buyer has a single opportunity to bid, she would be more likely to bid at or above the ask price.}.
Of course there are indirect costs associated to each failed transaction: (i) the time `{\em wasted}' is penalised by discounting and (ii) the agent must stick with her decision between time $\tau$ and time $\tau+\vartheta$, hence potentially missing other investment opportunities in that period (see Remark \ref{rem:theta}).

In the sequel we present a few examples of application of our model. We should remark that while in such examples the decision maker optimises future expected cashflows, it is possible to envisage models in which $\varphi$ and $\psi$ represent utility functions.

\subsubsection*{Pricing of real options} Consider an investor who bids to secure a certain investment opportunity. According to the traditional literature on irreversible investment the net present value (conditional on $\cF_t$) of future discounted cashflows resulting from the project follows a stochastic dynamics $(X_t)_{t\ge 0}$. The investor must sustain a sunk cost $K$ at the time $\tau$ of entering the project and she then receives $X_\tau$. Traditionally (see, e.g., \citet{dixit1994}) this problem is cast as
\[
\sup_\tau\E_x\big[e^{-r\tau}(X_\tau-K)\big].
\]
We can imagine that the sunk cost $K$ results from the investor's bid to secure the project. Then, the investor could choose to bid high and secure the project with certainty or to bid low and hope that her bid beats any potential competitors. Without entering into the complex realm of stochastic games we can say that a high bid corresponds to a `large' sunk cost $K_1>0$. On the contrary, a low bid corresponds to a `small' sunk cost $K_2<K_1$ but it is also associated to a probability of success $p\in(0,1)$. If the attempt is unsuccessful, the investor will wait for the next viable opportunity, which we model with the recursive structure of the problem. That is, in \eqref{eq:v} we could consider $\varphi(x)=(x-K_1)$ and $\psi(x)=(x-K_2)$. In this context, the delay $\vartheta$ associated to the outcome of the low bid can be understood as the the time it takes to receive and scrutinise several bids from competing firms. For the large bid $K_1$ such process is unnecessary and the project is assigned more quickly.

\subsubsection*{R\&D decisions} The R\&D department of a large firm (e.g., mobile phone makers) is tasked with developing new technologies that would allow the firm to expand its market share. The R\&D can pick two directions of work: (i) they can attempt to develop an advanced and disruptive technology that potentially allows the company to get a large market share $M\in(0,1)$ or (ii) they can develop standard upgrades of their current technology which guarantee to the company a smaller market share $0<m<M<1$. The development of a more advanced technology is risky and the estimated probability of success is $p\in (0,1)$. Moreover, developing a technology takes time, which implies that the outcome of such an investment is uncertain and comes with a delay. The standard upgrades of the existing technology instead are routinely performed and with no delay.

The whole market generates revenues at a rate modelled by a Markov process $(X_t)_{t \ge 0}$. So, at any given time $t$, the whole market's expected future discounted revenues read
\[
f(X_t):=\E\Big[\int_t^\infty e^{-rs}X_s\ud s\Big|\cF_t\Big].
\]
If the R\&D opts for the safer project the firm receives $m f(X_\tau)$ with certainty. If instead they opt for the more disruptive project the firm receives $M f(X_{\tau+\vartheta})$ at a future time and with probability $p$. Clearly, this corresponds to $\varphi(x)=m f(x)$ and $\psi(x)= M f(x)$ in \eqref{eq:v}.

\subsubsection*{Trading with the dark pool} The third example deals with a stock trading problem in the standard exchange and a dark pool. For this problem we develop a detailed mathematical model and present its solution in Section \ref{sec:appl}. Hence, here we only provide a brief overview of the related literature, in order to contextualise the problem.

Dark pools are trading venues that offer investors an alternative to the standard exchange (i.e., the lit market). Trading rules may vary across various dark pools but they share some common features. In dark pools information about outstanding orders (for instance, prices and market depth\footnote{``volume posted in the limit order book and available for immediate execution'', see \citet{cartea2015algorithmic}.}) is not available until the trade has occurred (as opposed to the readily available information in the lit market). Therefore, each trader's interest remains hidden from the rest of the market unless the trade actually occurs. Moreover, orders that are {\em successful}, are executed at a more favourable price, compared to the traditional lit market (for instance, the midpoint between the best bid and the best ask in the lit market); see, e.g., \citet{mittal2008you} and \citet{degryse2009shedding} for further details on specific features of various dark pools.
Due to the fact that available liquidity is not displayed, the execution of orders is uncertain: buy and sell orders are crossed as they arrive into the system, resulting in a delayed information flow.  One of the main advantages of trading in dark pools is that price impact negligible. For this reason these venues were originally used by, e.g., institutional investors, who typically deal with large-volume trades.
Popularity of dark pools has rapidly increased and nowadays they represent a consistent percentage of the trading volume in the US equity market (more than 15\%, see, e.g. \citet{ganchev2010censored}, \citet{zhu2014dark}, \citet{buti2017dark}). Over time the characteristics of traders in dark pools have changed and it is not unusual to find small orders whose execution is used, for instance, to detect if larger orders have been posted or as a proxy for the pool's state of liquidity. The success of dark pools has motivated the recent interest from academics on several aspects, such as: the impact of information leakage, adverse effects on market quality and optimally distribution of large orders in different trading venues.

Several papers consider the problem of a trader who can invest in the standard exchange and in a dark pool. This literature is mainly related to problems of optimal liquidation and aims to study features like the effect of liquidity and the impact on prices of the optimal liquidation strategy (see, e.g.~\citet{kratz2015portfolio}, \citet{kratz2018optimal}, \citet{crisafi2016simultaneous} and references therein). An in-depth analysis of trading mechanisms in dark pools falls outside the scopes of our paper. Instead, we suggest a simple model that draws on the class of recursive stopping problems studied in this work, with the aim to inform future more focussed applications.
In particular, in our example an investor holds a certain number of shares of a stock and wants to find the best time to sell the whole inventory with a single trade.
In line with existing literature (see, e.g., \citet{kratz2015portfolio}, \citet{cartea2015algorithmic} and \citet{boni2013dark}) we model orders in the dark pool as {\em complete-or-zero-execution} so that unexecuted (or {\em resting}) orders are held in the system until a matching order arrives or until cancellation\footnote{Some types of dark pools accept {\em immediate-or-cancel} orders, which corresponds to the case of $\vartheta=0$ in \eqref{eq:v}.} (e.g., at the end of the day, or end of the hour).
We consider a stochastic two-factor model $(S,K)$ where $S$ is the stock price in the lit market and $S+K$ is the price in the dark pool. The assumption of a two-factor model, with  stochastic spread $K$, is in line with the recent literature (see, e.g., \citet{crisafi2016simultaneous}) and it is closer to reality than one-factor models with deterministic or constant spread (recall that in fact the price in the dark pool is often the mid price between bid and ask as quoted on the standard exchange).

\subsection{Structure of the paper}
The paper is organised as follows. In Section \ref{sec:setting} we give the main modelling assumptions, we introduce the recursive optimal control/stopping problem and we establish its equivalence to a recursive problem of optimal stopping only. In Section \ref{sec:value} we prove that the stopping problem is well-posed and that an optimal stopping time exists. The application to trading in the dark pool, using a two-dimensional geometric Brownian motion, is illustrated in Section \ref{sec:appl}. In particular, the optimal trading boundaries and regularity properties of the value function are given in Section \ref{sec:boundaries}. Finally, a short technical Appendix concludes the paper.

\section{Modelling framework and problem formulation}\label{sec:setting}
We fix a probability space $(\Omega, \cF, \P)$ endowed with a right-continuous and complete filtration $\bF=(\cF_t)_{t \geq 0}$ with $\cF_\infty=\bigvee_{t \geq 0} \cF_t=:\cF$. Let  $X=(X_t)_{t \geq 0}$  be a right-continuous, strong Markov process, taking values in $\R^d$, that can be realised as a stochastic flow $(t,x) \mapsto X^x_t$, so that $X^x_0=x$, $\P$-a.s. We also assume that $X$ is quasi left-continuous, i.e., left continuous over stopping times. We denote $\P_x(\,\cdot\,):=\P(\,\cdot\,|X_0=x)$ and, for any integrable functional $f$ on the space of right-continuous paths in $\R^d$, we denote $\E_x[f(X_\cdot)]=\E[f(X^x_\cdot)]=\E[f(X_\cdot)|X_0=x]$. Moreover, thanks to strong Markov property we can also write $\E_{X_{\tau}}[f(X_\cdot)]=\E[f(X_{\tau+\cdot})|\cF_{\tau}]$ for any $\bF$-stopping time $\tau$. Finally, we let $\vartheta$ be a non-negative random variable, independent of $X$, with cumulative distribution $F(\cdot)$.

In what follows we consider a constant discount factor $r>0$, a parameter $p \in (0,1)$ and functions $\varphi:\R^d \to \R_+$ and $ \psi: \R^d \to \R_+$ with $\varphi\leq \psi$ on $\R^d$. We denote by $\cT$ the set of $\bF$-stopping times and define the set of admissible control/stopping pairs as
\begin{align*}
\D=\{(\tau, \alpha): \ \tau \in \cT, \ \alpha\in \{0,1\}, \ \alpha \mbox{ is } \cF_\tau\mbox{-measurable}\}.
\end{align*}

\begin{remark}\label{rem:finite_time_horizon}
The set $\cT$ may be either bounded or unbounded, in order to accommodate both finite-time and infinite-time horizon problems. If $\cT$ is bounded we will assume that the first coordinate of the $d$-dimensional process $X$ is `time', i.e., $X_t=(t,X^1_t,\ldots X^{d-1}_t)$. All the results presented in Sections \ref{sec:setting} and \ref{sec:value} hold for both the finite-time and infinite-time horizon problem. Only the proof of Lemma \ref{lem:usc-G} requires a small tweak, which is discussed in detail.
\end{remark}

\begin{remark}
Since we have no restrictions on the probability distribution of the delay, $\vartheta$, our results also include the case of bounded delay (that is, $F(\ud t)$ has compact support) and deterministic delay (that is, $F(\ud t)=\delta_{t_0}(\ud t)$ for some $t_0>0$).
\end{remark}

Let $|\, \cdot \,|_d$ denote the Euclidean norm in $\R^d$, let $x_i$ be the $i$-th coordinate of $x\in\R^d$ and if $\cT$ is unbounded we adopt the convention that
\[
f(X^x_\tau)\I_{\{\tau=\infty\}}=\limsup_{t\to \infty}f(X^x_t),\qquad\text{$\P$-a.s.},
\]
for any Borel-measurable function $f:\R^d\to\R$, each $x\in\R^d$ and any $\tau\in\cT$;  such convention is clearly unnecessary in the case of finite-time horizon.

\subsection{Infinite horizon problem}\label{sec:infinite}
In order to set out the notation, it is convenient to first introduce the problem with infinite-time horizon. The one with finite-time horizon will be presented in Section \ref{sec:finite} and only requires small modifications.

Our objective is to solve the following problem.
\vspace{+8pt}

\noindent{\bf Problem 1.} Find a continuous function $v:\R^d\to \R_+$  that satisfies
\begin{align}\label{eq:value}
v(x)=\!\sup_{(\tau, \alpha)\in \D}\!\E_x\!\left[e^{-r \tau} \varphi(X_\tau) \I_{\{\alpha=0\}} \!+\! e^{-r (\tau+\vartheta) }   \left( p \psi(X_{\tau+\vartheta})\!+\! (1\!-\!p) v(X_{\tau+\vartheta})\right) \I_{\{\alpha=1\}} \right].
\end{align}
\vspace{+8pt}

The optimisation problem in \eqref{eq:value} describes situations in which an agent `stops' at time $\tau$ and chooses between an immediate payoff $\varphi(X)$, if $\alpha=0$, or a larger payoff $\psi(X)$, if $\alpha=1$, which will only be attained with probability $p\in(0,1)$ at a future random time $\tau+\vartheta$. If the agent opts for $\psi(X)$ and the payoff is not attained (which occurs with probability $1-p$) then the optimisation must start afresh (at the time $\tau+\vartheta$ when the outcome is revealed).

It is intuitively clear that in choosing her strategy the agent will need to keep track of multiple sources of uncertainty. As usual there is an underlying stochastic dynamic $X$ and a discount factor that penalises waiting. Additionally to that, one must account for the relative convenience of $\psi$ compared to $\varphi$, which needs to be `weighted' with the risk of an unsuccessful transaction and the random waiting time after the decision to stop.

\begin{remark}[The delay $\vartheta$]\label{rem:theta}
If at time $\tau$ the agent commits to $\alpha=1$, she must stick with her decision until time $\tau+\vartheta$. This is in line for example with the situation of a research team that submits a grant proposal: until the funder makes a decision (which may happen within a period of time that is more or less known) the team would not withdraw the proposal nor submit to another funder. Other examples of (irreversible) investment with time-lag (but no recursive structure) can be found for instance in \cite{bar1996investment}.

Our model also allows the agent to set a maximum (deterministic) waiting time $t_0$. Indeed we can take $\vartheta=t_0\wedge \gamma$ for some random variable $\gamma\ge 0$ independent of $X$ (i.e., we can interpret $\vartheta$ as the smallest between $t_0$ and the time $\gamma$ at which the outcome of the risky project is actually revealed).
As explained in the Introduction, such specification is natural for our application to trading: unexecuted orders in the dark pool are held in the system until a matching order arrives (i.e., until $\gamma$) or until cancellation (i.e., until $t_0$); see, e.g., \cite{kratz2015portfolio}, \cite{cartea2015algorithmic} and \cite{boni2013dark}.
\end{remark}

\begin{remark}[Extensions and standard optimal stopping]
\begin{itemize}
\item[ ]
\item[(a)] The problem formulation above may be extended to accommodate specific applied situations. While it is difficult to account concisely for all such possible extensions, we note that in \eqref{eq:value} one could add a fixed cost $c>0$ that further penalises the negative outcome in case $\alpha=1$, by taking $(1-p)( v(X^x_{\tau+\vartheta})-c)$. This tweak does not affect the analysis and the results in the rest of the paper and we set $c=0$ for simplicity.
\item[(b)]
If we take  $p=0$ and $\P(\vartheta=0)=1$ we reduce to a {\em classical optimal stopping problem} with gain function $\varphi$. Then equation \eqref{eq:value} can be interpreted as a version of the dynamic programming principle, where at each stopping time $\tau$ the optimiser can decide whether to stop ($\alpha=0$) or to continue ($\alpha=1$).
\end{itemize}
\end{remark}

We now prove that Problem 1 has an alternative characterisation in terms of a problem of optimal stopping only. To this end we introduce the following optimisation problem.
\vspace{+8pt}

{\bf Problem 2.} Find a continuous function $\tilde v:\R^d\to \R_+$  that satisfies
\begin{align}\label{eq:vtilde}
\tilde{v}(x)=\sup_{\tau\in \cT}\E_x\left[e^{-r \tau} \max\left\{\varphi(X_\tau),(\Lambda\tilde{v})(X_\tau) \right\}\right],
\end{align}
where for any continuous function $f:\R^d\to \R_+$ we define
\begin{align}\label{eq:Lambda}
(\Lambda f)(x):=\int_0^\infty e^{-rt}\E_x\left[p\psi(X_t)+(1-p)f(X_t)\right]F(\ud t).
\end{align}
\vspace{+8pt}

\begin{lemma}\label{lem:value}
A continuous function $v:\R^d\to \R_+$ is a solution of {\bf Problem 1} if and only if it solves {\bf Problem 2}. Moreover, if {\bf Problem 2} has a solution and admits an optimal stopping time $\tau_*$, then the couple $(\tau_*,\alpha^*)$, with $\alpha^*:=\I_{\{(\Lambda v)(X_{\tau_*})>\varphi(X_{\tau_*})\}}$, is optimal for {\bf Problem 1}.
\end{lemma}
\begin{proof}
Assume $v$ is a solution of {\bf Problem 1}. From \eqref{eq:value}, using independence of $\vartheta$ and $X$ we obtain
\begin{align}
\label{eq:last}
v(x)& = \sup_{(\tau, \alpha)\in \D}\E \bigg[e^{-r \tau} \varphi(X^x_\tau) \I_{\{\alpha=0\}}\! \\
&\qquad\qquad\qquad+\!\! \int_0^{\infty}\!\!\!\!e^{-r (\tau+t) }   \left( p \psi(X^x_{\tau+t})\!+\! (1-p) v(X^x_{\tau+t})\right) F(\ud t) \I_{\{\alpha=1\}} \bigg].
\end{align}
Since $\alpha$ is $\cF_\tau$-measurable, using Fubini's theorem, the strong Markov property of $X$ and \eqref{eq:Lambda} we get
\begin{align}\label{eq:tower}
&\E \left[\int_0^{\infty}e^{-r (\tau+t) } \left( p \psi(X^x_{\tau+t})+ (1-p) v(X^x_{\tau+t})\right) F(\ud t)\I_{\{\alpha=1\}}\Big|\cF_{\tau}\right]\\
&=\int_0^{\infty}e^{-r (\tau+t) }  \E\left[ \left( p \psi(X^x_{\tau+t})+ (1-p) v(X^x_{\tau+t})\right) |\cF_\tau\right] F(\ud t)\I_{\{\alpha=1\}}\\
&= \int_0^{\infty}e^{-r (\tau+t) }  \E_{X^x_\tau}\left[ \left( p \psi\big(X_{t}\big)+ (1-p) v\big(X_t\big)\right)  \right] F(\ud t) \I_{\{\alpha=1\}}= e^{-r\tau} (\Lambda v) (X^x_\tau)\I_{\{\alpha=1\}}.
\end{align}

Now, we can use the tower property of conditional expectation and \eqref{eq:tower}, in the right-hand side of \eqref{eq:last}, in order to obtain
\begin{align}\label{eq:ref}
v(x)&= \sup_{(\tau, \alpha)\in \D}\E\left[e^{-r \tau} \varphi(X^x_\tau) \I_{\{\alpha=0\}} +  e^{-r\tau} (\Lambda v) (X^x_\tau) \I_{\{\alpha=1\}}  \right]\\
&\le  \sup_{\tau \in \cT}\E\left[e^{-r \tau} \max\left\{ \varphi(X^x_\tau),  (\Lambda v) (X^x_\tau) \right\} \right].
\end{align}
Equality in \eqref{eq:ref} is obtained by choosing the Markovian control $\alpha(x)=\I_{\{(\Lambda v)(x)>\varphi(x)\}}$. Optimality of the couple $(\tau_*,\alpha^*)$ then follows as well.

Finally, for the {\em only if} part of the statement, we can reverse the argument above and show that a solution $\tilde v$ of {\bf Problem 2} must satisfy \eqref{eq:last}.
\end{proof}

Lemma \ref{lem:value} allows us to use equivalently the problem formulation given in either \eqref{eq:value} or \eqref{eq:vtilde}. In the rest of the paper we will mainly focus on the study of \eqref{eq:vtilde} and we set $\tilde v=v$ throughout.

Next we introduce the set
\begin{align}\label{eq:A}
\cA_d:=\left\{f: f\in C(\R^d;\R_+),\:\text{such that}\:\|f\|_{\cA_d}<\infty\right\},
\end{align}
where
\begin{align}\label{eq:norm}
\|f\|^2_{\cA_d}:=\sup_{x\in\R^d}\frac{|f(x)|^2}{1+|x|_d^2}.
\end{align}
It is not difficult to see that $(\cA_d, \| \cdot \|_{\cA_d})$ is a Banach space (the proof of this fact is given in Appendix for completeness).

\begin{remark}
The case in which $X$ is a two-dimensional geometric Brownian motion will be considered in Section \ref{sec:appl}. In that setting the process is bound to evolve in $\R^2_+$ and we will consider the space $\cA^+_2$ defined as in \eqref{eq:A} but with $\R^2_+$ in place of $\R^d$.
\end{remark}

Next we give standing assumptions on the process $X$ and on the payoff functions.

\begin{assumption}\label{ass:sub}
The stochastic flow $x\mapsto X^x_t$ is $\P$-a.s.~continuous for all $t\ge 0$. Moreover,
\begin{itemize}
\item[(i)]  there exists $\rho\in (0,1)$ such that the process $(\widehat X_t)_{t\geq 0}$ defined by
\[
\widehat X_t:= e^{-2r(1-\rho)t} (1+|X_t|_d^2),\quad\text{for $t \geq 0$},
\]
is a $\P_x$-supermartingale for any $x\in\R^d$;
\item[(ii)] for any compact $K\subset\R^d$ we have
\begin{align}
 \sup_{x\in K}\E_x\left[\sup_{t \geq 0} e^{-r t} |X_t|_d\right]<\infty;
\end{align}
\item[(iii)] for any $x\in \R^d$ and  any sequence $(x_n)_{n \geq 0}$ converging to $x$, it holds
\begin{align}
\lim_{n \to \infty}\E\left[\sup_{t \geq 0} e^{-rt} |X^{x_n}_t-X^x_t|_d\right]=0; \label{ass:conv}
\end{align}
\item[(iv)] functions $\varphi$ and $\psi$ belong to $\cA_d$ (with $\varphi\le \psi$).
\end{itemize}
\end{assumption}

Continuity of the flow $x\mapsto X^x$ can be relaxed (see Remark \ref{rem:flow}) but it is convenient for a clear exposition. Also, sufficient conditions for the existence of a continuous modification of a random field are provided by the well-known Kolmogorov's continuity theorem (see, e.g., \cite[Thm.~2.8, Ch.~2]{KS}). Notice that if $X$ is a solution to a stochastic differential equation whose coefficients have sublinear growth, we can always find a constant $r>0$ sufficiently large to guarantee that (i) and (ii) in Assumption \ref{ass:sub} hold.  If moreover the coefficients are Lipschitz continuous, then (iii) also holds for suitable $r>0$ (these claims can be verified adapting the proofs of \cite[Thm.~9 and Cor.~10, Ch.~2, Sec.~5]{Krylov} to $e^{-rt}X_t$).

In Section \ref{sec:value}, we will often use that since the process $\widehat X$ in (i) of Assumption \ref{ass:sub} is a non-negative supermartingale, then it is a supermartingale for $t\in[0,\infty]$ and the optional sampling theorem gives
\begin{align}\label{eq:Optsam}
\E_x\big[\widehat X_{\tau}\big]\le (1+|x|^2_d), \quad\text{for any $\tau\in\cT$ and $x\in\R^d$}
\end{align}
(see, e.g., \cite[Prob.~3.16 and Thm.~3.22, Ch.~1, Sec.~3]{KS2}). For future reference we also notice that, for any $f\in\cA_d$, using (i) we have
\begin{align}\label{eq:unifint}
\sup_{x\in K}\E_x[e^{-2rt}f^2(X_t)]<\infty,\quad\text{for any compact $K\subset\R^d$ and any $t\ge 0$}.
\end{align}
Hence (see, e.g., \cite[Lem.~3, Ch.~2, Sec.~6]{shiryaev1988probability})
\begin{align}\label{eq:ui0}
\text{the family $\{e^{-rt}f(X^x_t),\:\:x\in K\}$ is uniformly integrable for any $t\ge 0$}.
\end{align}
Moreover, by continuity of the flow $x\mapsto X^x_t$ and of $f$, for any sequence $(x_n)_{n\ge 1}$ converging to $x\in\R^d$ and any $t\ge 0$ we have
\begin{align}\label{eq:ui2}
\lim_{n\to\infty}|f(X^x_t)-f(X^{x_n}_t)|=0,\qquad \P-a.s.
\end{align}
Then by \cite[Thm.~4, Ch.~2, Sec.~6]{shiryaev1988probability}, using \eqref{eq:ui0} we also have
\begin{align}\label{eq:ui3}
\lim_{n\to\infty}\E\left[e^{-rt}\left|f(X^x_t)-f(X^{x_n}_t)\right|\right]=0.
\end{align}

\subsection{Finite horizon problem}\label{sec:finite}

Here we formulate the finite-time horizon versions of {\bf Problem 1} and {\bf Problem 2}. We fix $T<\infty$ and for $z:=(t,x)\in[0,T]\times\R^{d-1}$ we denote $X_s=(t+s,X^1_s,\ldots X^{d-1}_s)$ for $s\in[0,T-t]$, under the measure $\P_{z}$ (or equivalently we denote the associated flow by $X^z$ with $X^z_0=z$). Further, we say that $\tau\in\cT$ if it is a stopping time and $\tau\in[0,T-t]$, $\P_z$-a.s.

Then, the analogue of {\bf Problem 1} in this setting reads:
\vspace{+4pt}

\noindent{\bf Problem 1'}. Find a continuous function $v:[0,T]\times\R^{d-1}\to \R_+$  that satisfies
\begin{align}\label{eq:value-T1}
v(z)\!=\!\sup_{(\tau, \alpha)\in \D}\!\E_z\!\left[e^{-r \tau} \varphi(X_\tau) \I_{\{\alpha=0\}} \!+\! e^{-r (\tau+\vartheta) } \!  \left( p \psi(X_{\tau+\vartheta})\!+\! (1\!-\!p) v(X_{\tau+\vartheta})\right)\I_{\{\tau+\vartheta\le T-t\}} \I_{\{\alpha=1\}} \right].
\end{align}
\vspace{+4pt}

Similarly, the analogue of {\bf Problem 2} reads:
\vspace{+4pt}

\noindent{\bf Problem 2'}. Find a continuous function $\tilde v:[0,T]\times\R^{d-1}\to \R_+$  that satisfies
\begin{align}\label{eq:value-T2}
\tilde v(z)\!=\!\sup_{\tau\in \cT}\E_z\left[e^{-r \tau} \max\left\{\varphi(X_\tau),(\Lambda \tilde v)(X_\tau) \right\}\right],
\end{align}
where for any continuous function $f:[0,T]\times\R^{d-1}\to \R_+$ we now define
\begin{align}\label{eq:Lambda-T}
(\Lambda f)(z):=\int_0^{T-t} e^{-rt}\E_z\left[p\psi(X_s)+(1-p)f(X_s)\right]F(\ud s).
\end{align}
\vspace{+4pt}

Notice that the integral in \eqref{eq:Lambda-T} is only up to $T-t$, due to the presence of the indicator of the event $\{\tau+\vartheta\le T-t\}$ in the formulation of {\bf Problem 1'}. Lemma \ref{lem:value} continues to hold for {\bf Problem 1'} and {\bf Problem 2'}, with the same proof. Hence we set $v=\tilde v$ throughout the paper.

In this framework Assumption \ref{ass:sub} and the subsequent discussion are understood to hold for $t\in[0,T]$. Moreover, condition (i) in that assumption only applies to the `spatial' part of the process, i.e., to the vector $(X^1,\ldots X^{d-1})$. Functions $f\in\cA_d$ are continuous from $[0,T]\times\R^{d-1}$ to $\R_+$ and the norm $\|\cdot\|_{\cA_d}$ is understood as
\begin{align*}
\|f\|^2_{\cA_d}:=\sup_{(t,x)\in[0,T]\times\R^{d-1}}\frac{|f(t,x)|^2}{1+|x|_{d-1}^2}.
\end{align*}

We will need $\Lambda f\in\! C([0,T]\times\R^{d-1})$, according to Lemma \ref{lem:Lambda} below. Then, for the finite-time horizon set-up it is convenient to make the next assumption.
\begin{assumption}\label{ass:finite-T}
In the finite-time horizon problem we have $F$ continuous on $[0,T]$ with $F(0)$ possibly strictly positive (i.e., the law of $\vartheta$ may have an atom at zero).
\end{assumption}

To conclude, notice that for $z=(T,x)$ we have $\tau=0$ and $\alpha$ is $\cF_0$-measurable. Then it is easy to see that
\begin{align}\label{eq:v(z)}
v(z)=&\sup_{\alpha\in\{0,1\}}\big[\varphi(z)\P_z(\alpha=0)+
\left(p\psi(z)+(1-p)v(z)\right)F(0)\P_z(\alpha=1)\big]\\
=&\max\{\varphi(z),\left(p\psi(z)+(1-p)v(z)\right)F(0)\}.\notag
\end{align}
If $F(0)=0$ we have $v(z)=\varphi(z)$. If instead $F(0)>0$, we have
\[
v(z)=\left(p\psi(z)+(1-p)v(z)\right)F(0)\quad\text{if and only if}\quad v(z)=\frac{pF(0)\psi(z)}{1-(1-p)F(0)}.
\]
Then, substituting back into \eqref{eq:v(z)} we get
\begin{align}\label{eq:vT}
v(T,x)=\max\left\{\varphi(T,x),\psi(T,x)\frac{pF(0)}{1-(1-p)F(0)}\right\}.
\end{align}

\section{Existence of a value}\label{sec:value}
In this section we prove that {\bf Problem 2} and {\bf Problem 2'} (hence {\bf Problem 1} and {\bf Problem 1'}) are well-posed. That is, the value function $v$ is uniquely determined as a fixed point in $\cA_d$ and an optimal pair $(\tau_*,\alpha^*)$ exists, thanks to Lemma \ref{lem:value}. In order to avoid repetitions, we present most of our analysis using the notation of the infinite-time horizon setting (Section \ref{sec:infinite}) but all the results hold with finite-time horizon and details are provided in all proofs as necessary.

Let us start by introducing the operator $\Gamma$ given by
\begin{align}\label{eq:gamma}
(\Gamma f)(x):=\sup_{\tau\in\cT}\E_x\left[e^{-r \tau} \max\left\{\varphi(X_\tau),(\Lambda f)(X_\tau) \right\}\right]
\end{align}
for every continuous function $f : \R^d \to \R_+$, where $\Lambda$ is defined in \eqref{eq:Lambda}. Equation \eqref{eq:gamma} defines an optimal stopping problem for each $f\in C(\R^d;\R_+)$.

Our goal is to prove Theorem \ref{thm:prob1} below. If the time horizon is $T<\infty$, we understand all the results to hold for $t\in[0,T]$ but we omit further notation for simplicity (see Section \ref{sec:finite}).
\begin{theorem}\label{thm:prob1}
Problem 2 (Problem 2') admits a unique solution $v\in\cA_d$. Moreover, the stopping time
\begin{align}\label{eq:tau*}
\tau_*=\inf\big\{s \geq 0: v(X_s)=\max\left[\varphi(X_s),(\Lambda v)(X_s) \right]\big\}
\end{align}
is optimal for \eqref{eq:vtilde}, the process
\[
t\mapsto e^{-rt}v(X_t),\qquad t\in[0,\infty]
\]
is a right-continuous (non-negative) $\P_x$-supermartingale and the process
\[
t\mapsto e^{-r(t\wedge\tau_*)}v(X_{t\wedge\tau_*}),\qquad t\in[0,\infty)
\]
is a right-continuous (non-negative) $\P_x$-martingale, for any $x\in\R^d$.
\end{theorem}

Notice that if $T<\infty$ then, given $z=(t,x)$, the infimum in \eqref{eq:tau*} is taken on $s\in[0,T-t]$ so that $\tau_*\le T-t$, $\P_z$-a.s.~as needed. The same comment applies to \eqref{eq:tau} in Lemma \ref{lem:OS} below.

The proof of Theorem \ref{thm:prob1} requires intermediate steps in order to show that the operator $\Gamma$ is a contraction in $\cA_d$. First we show in Lemma \ref{lem:Lambda} that the operator $\Lambda$ maps $\cA_d$ into itself. Second we prove in Lemma \ref{lem:OS} that an optimal stopping time exists in \eqref{eq:gamma} and that $\Gamma f$ is lower semi-continuous for each $f\in \cA_d$. Finally we show in Lemma \ref{lem:usc-G} that $\Gamma f$ is also upper semi-continuous for each $f\in \cA_d$, and hence continuous. The section ends with the proof of the contraction property of $\Gamma$.

\begin{lemma}\label{lem:Lambda}
For every $f \in \cA_d$ it holds that $\Lambda f \in \cA_d$.
\end{lemma}

\begin{proof}
Consider the infinite-time horizon problem. First, for every function $f \in \cA_d$, we have that $\Lambda f\ge 0$. Moreover
\begin{align}
|(\Lambda f)(x)|&
\leq\int_0^\infty e^{-rt}\left( p \E_x\big[ |\psi(X_t)|\big] +(1-p) \E_x\big[|f(X_t)|\big]\right)F(\ud t)\\
&\leq\big(p \|\psi\|_{\cA_d}+(1-p) \|f\|_{\cA_d}\big)\int_0^\infty e^{-rt} \E\left[(1+|X^x_t|^2_d)^\frac{1}{2}\right]F(\ud t)\\
&\leq ( p \|\psi\|_{\cA_d}+(1-p) \|f\|_{\cA_d}) (1+|x|^2_d)^{\frac{1}{2}}
\end{align}
where we first used triangular inequality and then, in the final step, we used Jensen's inequality and condition (i) in Assumption \ref{ass:sub}.
Consequently $\|\Lambda f\|_{\cA_d}<\infty$.

Since the flow $x\mapsto X^x_t$ is continuous, we can use dominated convergence, continuity of $\psi$ and $f$ and \eqref{eq:ui3} to conclude
\begin{align}
\lim_{n \to \infty}(\Lambda f) (x_n)=(\Lambda f) (x).
\end{align}
The proof is identical in the finite-time horizon case, where we use $\Lambda f$ as in \eqref{eq:Lambda-T}.
\end{proof}

\begin{lemma}\label{lem:OS}
For every $f\in \cA_d$, the stopping problem in \eqref{eq:gamma} is well-posed in the sense that
\begin{align}\label{eq:tau}
\tau_*^f=\inf\big\{s \geq 0: (\Gamma f)(X_s)=\max\left[\varphi(X_s),(\Lambda f)(X_s) \right]\big\}
\end{align}
is an optimal stopping time, the function $\Gamma f$ is lower semi-continuous, the process
\begin{align}\label{super}
t \mapsto e^{-rt} (\Gamma f)(X_t), \quad t \in [0,\infty]
\end{align}
is a right-continuous (non-negative), $\P_x$-supermartingale and the stopped process
\begin{align}\label{mg}
t \mapsto e^{-r(t\wedge \tau_*^f)} (\Gamma f)(X_{t\wedge \tau_*^f}), \quad t \in [0,\infty)
\end{align}
is a right-continuous (non-negative), $\P_x$-martingale, for any $x\in\R^d$.
\end{lemma}
\begin{proof}
Fix $f\in \cA_d$. By Lemma \ref{lem:Lambda} it is immediate to see that $x \mapsto \max\{\varphi(x), (\Lambda f) (x)\}$ is continuous and there exists a constant $c>0$ (depending on $\|f\|_{\cA_d}$, $\|\psi\|_{\cA_d}$ and $\|\varphi\|_{\cA_d}$) such that
\begin{align}\label{eq:max}
\max\{\varphi(x),(\Lambda f)(x)\}\le c(1+|x|_d)
\end{align}
since $(1+|x|_d^2)^{1/2}\le 1+|x|_d$. By (ii) in Assumption \ref{ass:sub} and \eqref{eq:max} we get
\begin{align}\label{eq:ui}
\E_x\left[\sup_{t\ge 0} e^{-rt}\max\{\varphi(X_t),(\Lambda f)(X_t)\}\right]<\infty.
\end{align}
The assumption of continuity of the flow $x\mapsto X^x_t$, for $t\ge 0$, implies that $X$ is a Feller process. Then, combining Lemma 3 and Lemma 4 from \cite[Ch.~3]{shiryaev2007optimal} we obtain that $\Gamma f$ is lower semi-continuous.

Lower semi-continuity of $\Gamma f$, continuity of the payoff and \eqref{eq:ui} allow us to apply \citet[Cor.~2.9, Ch.~1, Sec.~2]{PS}. Then $\tau_*^f$ as in \eqref{eq:tau} is indeed optimal and the (super)-martingale properties \eqref{super} and \eqref{mg} of the discounted value process hold.
\end{proof}

\begin{lemma}\label{lem:usc-G}
For every $f\in\cA_d$ and $x\in \R^d$ given and fixed, we have
\begin{align}\label{eq:usc}
\limsup_{n\to\infty}(\Gamma f)(x_n)\le (\Gamma f)(x)
\end{align}
for any sequence $(x_n)_{n\ge 1}$ such that $x_n\to x$ as $n\to\infty$.
\end{lemma}
\begin{proof}
We first address the problem with infinite-time horizon and then the one with finite-time horizon.

{\bf Step 1}. ({\em Infinite-time horizon}.)
Fix $f\in\cA_d$, $x\in\R^d$ and let $(x_n)_{n\ge 1}$ be a sequence such that $x_n\to x$ as $n\to\infty$. With no loss of generality we can assume $|x_n|_d\le 1+|x|_d$ for $n\ge 1$. In order to simplify the notation we set $G(x):=\max\{\varphi(x),(\Lambda f)(x)\}$, so that $G\in\cA_d$ by Lemma \ref{lem:Lambda}.

Thanks to Lemma \ref{lem:OS}, for any $x_n$ there exists an optimal stopping time $\tau_n:=\tau_*^f(x_n)$ for the problem in \eqref{eq:gamma} with value function $(\Gamma f)(x_n)$. Take an arbitrary deterministic time $S>0$, then we have
\begin{align}\label{eq:usc1}
&(\Gamma f)(x_n)-(\Gamma f)(x)\\
&\le \E\left[e^{-r\tau_n}\left(G(X^{x_n}_{\tau_n})\!-\!G(X^{x}_{\tau_n})\right)\right]\\
&= \E\left[e^{-r\tau_n}\left(G(X^{x_n}_{\tau_n})\!-\!G(X^{x}_{\tau_n})\right)\I_{\{\tau_n\le S\}}\right]+\E\left[e^{-r\tau_n}\left(G(X^{x_n}_{\tau_n})\!-\!G(X^{x}_{\tau_n})\right)\I_{\{\tau_n>S\}}\right]
\end{align}
We need to consider the two terms in the last line separately.

For the second term, using Cauchy-Schwarz inequality, the growth condition on $G\in\cA_d$ and (i) in Assumption \ref{ass:sub} (see also \eqref{eq:Optsam}) we obtain
\begin{align}\label{eq:usc2}
&\E\left[e^{-r\tau_n}\left(G(X^{x_n}_{\tau_n})\!-\!G(X^{x}_{\tau_n})\right)\I_{\{\tau_n>S\}}\right] \\ & \le \E\left[e^{-2r \rho \tau_n}\I_{\{\tau_n>S\}}\right]^{\frac{1}{2}}\E\left[e^{-2r(1-\rho)\tau_n}\left(G(X^{x_n}_{\tau_n})-G(X^{x}_{\tau_n})\right)^2\right]^{\frac{1}{2}}\\
&\le \sqrt{2}\|G\|_{\cA_d} e^{-r\rho S}\E\left[e^{-2r(1-\rho)\tau_n}\left(2+|X^{x_n}_{\tau_n}|^2_d+|X^{x}_{\tau_n}|^2_d\right)\right]^{\frac{1}{2}}\\
&\le  c_1 (1+|x|_d)\|G\|_{\cA_d} e^{-r\rho S},
\end{align}
where in the final inequality we have used that $|x_n|_d\le 1+|x|_d$ and $(1+|x|_d^2)^{\frac{1}{2}}\le (1+|x|_d)$. Notice that the constant $c_1>0$ is independent of $S$ and $n$.

Next we consider the first  term in the last line of \eqref{eq:usc1}. We fix $m\ge 1$ and define the stopping times
\[
\sigma^m_{n}:=\inf\{t\ge 0: |X^{x_n}_t|_d\vee|X_t^x|_d\ge m\}.
\]
Then we have that
\begin{align}\label{eq:usc3}
&\E\left[e^{-r\tau_n}\left(G(X^{x_n}_{\tau_n})\!-\!G(X^{x}_{\tau_n})\right)\I_{\{\tau_n\le S\}}\right]\\
&=\!\E\left[e^{-r\tau_n}\!\left(G(X^{x_n}_{\tau_n})\!-\!G(X^{x}_{\tau_n})\right)\I_{\{\tau_n\le S\}\cap\{\tau_n\le \sigma_n^m\}}\right]\\
&\quad+\!
\E\left[e^{-r\tau_n}\!\left(G(X^{x_n}_{\tau_n})\!-\!G(X^{x}_{\tau_n})\right)\I_{\{\tau_n\le S\}\cap\{\tau_n> \sigma_n^m\}}\right]
\end{align}
and we need to study separately the two terms
\[
A_1:=\!\E\left[e^{-r\tau_n}\!\left(G(X^{x_n}_{\tau_n})\!-\!G(X^{x}_{\tau_n})\right)\I_{\{\tau_n\le S\}\cap\{\tau_n\le \sigma_n^m\}}\right]
\]
and
\[
A_2:=\E\left[e^{-r\tau_n}\!\left(G(X^{x_n}_{\tau_n})\!-\!G(X^{x}_{\tau_n})\right)\I_{\{\tau_n\le S\}\cap\{\tau_n> \sigma_n^m\}}\right].
\]

For the first one we notice that, given an arbitrary $\eta>0$, there exists $\eps_{\eta,m}>0$ such that
\begin{align}
\sup|G(x)-G(y)|\le \eta,
\end{align}
where the supremum is taken over all $|x|_d\le m$, $|y|_d\le m$, such that $|x-y|_d\le \eps_{\eta,m}$. Moreover, due to (iii) in Assumption \ref{ass:sub}, for any given $\delta>0$ we can find $N_{\delta,S,\eta,m}\ge 1$ such that
\begin{align}\label{eq:usc4}
\P\left(\sup_{0\le t\le S}|X^{x_n}_t-X^x_t|>\eps_{\eta,m}\right)\le \delta,\quad\text{for all $n\ge N_{\delta,S,\eta,m}$}.
\end{align}
Set
\[
E_{n,S,\eta,m}:=\left\{\sup_{0\le t\le S}|X^{x_n}_t-X^x_t|>\eps_{\eta,m}\right\}
\]
and for simplicity denote $E=E_{n,S,\eta,m}$. Using Cauchy-Schwarz inequality, (i) in Assumption \ref{ass:sub} and estimates similar to those in \eqref{eq:usc2}, we obtain
\begin{align}\label{eq:usc5}
A_1=&\E\left[e^{-r\tau_n}\!\left(G(X^{x_n}_{\tau_n})\!-\!G(X^{x}_{\tau_n})\right)\I_{\{\tau_n\le S\}\cap\{\tau_n\le \sigma_n^m\}\cap E}\right]\\
&+\E\left[e^{-r\tau_n}\!\left(G(X^{x_n}_{\tau_n})\!-\!G(X^{x}_{\tau_n})\right)\I_{\{\tau_n\le S\}\cap\{\tau_n\le \sigma_n^m\}\cap E^c}\right]\\
\le & \E\left[e^{-2r\tau_n}\!\left(G(X^{x_n}_{\tau_n})\!-\!G(X^{x}_{\tau_n})\right)^2\right]^{\frac{1}{2}}\P\left(E\right)^{\frac{1}{2}}+\eta\\
\le &  \|G\|_{\cA_d} c_2 (1+|x|_d)\sqrt{\delta}+\eta,\qquad\text{for $n\ge N_{\delta,S,\eta,m}$}
\end{align}
where the constant $c_2>0$ is independent of $\delta$, $\eta$, $n$, $m$, $S$.

Likewise, for the other term we obtain
\begin{align}\label{eq:usc6}
A_2\le& \E\left[e^{-2r\tau_n}\!\left(G(X^{x_n}_{\tau_n})\!-\!G(X^{x}_{\tau_n})\right)^2\right]^{\frac{1}{2}}\P(\sigma^m_n<\tau_n\le S)^{\frac{1}{2}}\\
\le&  \|G\|_{\cA_d} c_3 (1+|x|_d)\P(\sigma^m_n< S)^{\frac{1}{2}}.
\end{align}
We now find an upper bound for $\P(\sigma^m_n< S)$. By sub-additivity of $\P$, Markov inequality and (ii) in Assumption \ref{ass:sub} we obtain
\begin{align}\label{eq:usc7}
\P(\sigma^m_n< S)\le& \P\left(\sup_{0\le t\le S }|X^{x_n}_t|_d>m\right)+\P\left(\sup_{0\le t\le S }|X^{x}_t|_d>m\right)\\
\le & \frac{1}{m}e^{rS}\left(\E\left[\sup_{0\le t\le S }e^{-rt}|X^{x_n}_t|_d\right]+\E\left[\sup_{0\le t\le S }e^{-rt}|X^{x}_t|_d\right]\right)\le \frac{1}{m}e^{rS}c_4.
\end{align}
Both the constants $c_3,c_4>0$ are independent of $n$, $m$, $S$ (since $x_n$ and $x$ lie in a compact).

Combining \eqref{eq:usc1}, \eqref{eq:usc2}, \eqref{eq:usc5}, \eqref{eq:usc6} and \eqref{eq:usc7} we get, for all $n\ge N_{\delta,S,\eta,m}$
\begin{align}\label{eq:usc8}
(\Gamma f)(x_n)-(\Gamma f)(x)\le c (1+|x|_d)\|G\|_{\cA_d} \left(e^{-r\rho S}+\sqrt{\delta}+e^{rS/2}/\sqrt{m}\right)+\eta,
\end{align}
where $c:=\max\{ c_i, i=1,\ldots, 4\}$.
Hence, in particular
\begin{align}\label{eq:usc9}
\limsup_{n\to\infty}(\Gamma f)(x_n)-(\Gamma f)(x)\le c (1+|x|_d)\|G\|_{\cA_d} \left(e^{-r\rho S}+\sqrt{\delta}+e^{rS/2}/\sqrt{m}\right)+\eta.
\end{align}
Keeping $S$ fixed and letting $\eta,\delta\to 0$ and $m\to\infty$ gives
\begin{align}\label{eq:usc10}
\limsup_{n\to\infty}(\Gamma f)(x_n)-(\Gamma f)(x)\le c (1+|x|_d)\|G\|_{\cA_d} e^{-r\rho S}.
\end{align}
Finally, letting $S\to\infty$ we obtain \eqref{eq:usc}.
\vspace{+4pt}

{\bf Step 2}. ({\em Finite-time horizon}.) Recall the set-up and notation from Section \ref{sec:finite}. Fix $z:=(t,x^1,\ldots x^{d-1})$ and take $z_n$ converging to $z$ as $n\to\infty$, with $z_n=(t_n,x^1_n,\ldots x^{d-1}_n)$. Then, a stopping time $\tau$ is admissible for $(\Gamma f)(z)$ provided that $\tau\le T-t$. Letting $\tau_n$ be optimal for $(\Gamma f)(z_n)$ we have that $\tau_n\wedge(T-t)$ is admissible for $(\Gamma f)(x)$. Using these stopping times, the inequality in \eqref{eq:usc1} changes to
\begin{align*}
&(\Gamma f)(z_n)-(\Gamma f)(z) \\
&\le \E\left[e^{-r\tau_n}\left(G(X^{z_n}_{\tau_n})\!-\!G(X^{z}_{\tau_n})\right)\I_{\{\tau_n\le T-t\}}\right]\\
&\quad+\E\left[\left(e^{-r\tau_n}G(X^{z_n}_{\tau_n})\!-\!e^{-r(T-t)}G(X^{z}_{T-t})\right)\I_{\{\tau_n>T-t\}}\right].
\end{align*}
The first term on the right-hand side above can be treated exactly as in step 1 with the event $\{\tau_n\le S\}$ therein replaced by $\{\tau_n\le T-t\}$. Hence it gives
\begin{align*}
\E\left[e^{-r\tau_n}\left(G(X^{z_n}_{\tau_n})\!-\!G(X^{z}_{\tau_n})\right)\I_{\{\tau_n\le T-t\}}\right]\le \eta+\|G\|_{\cA_d}c(1+|x|_d)\big(\sqrt{\delta}+e^{rT/2}/\sqrt{m}\big),
\end{align*}
for some constant $c>0$ and any $\delta,\eta>0$ and $m\ge 1$. For the second term we use that $T<\infty$ and that $|t-t_n|$ can be made arbitrarily small. First we write
\begin{align*}
\E&\left[\left(e^{-r\tau_n}G(X^{z_n}_{\tau_n})\!-\!e^{-r(T-t)}G(X^{z}_{T-t})\right)\I_{\{\tau_n>T-t\}}\right]\\
=&\E\left[e^{-r\tau_n}\left(G(X^{z_n}_{\tau_n})\!-\!G(X^{z}_{\tau_n})\right)\I_{\{\tau_n>T-t\}}+\left(e^{-r\tau_n}G(X^{z}_{\tau_n})\!-\!e^{-r(T-t)}G(X^{z}_{T-t})\right)\I_{\{\tau_n>T-t\}}\right]\\
\le &\E\left[e^{-r\tau_n}\left|G(X^{z_n}_{\tau_n})\!-\!G(X^{z}_{\tau_n})\right|\I_{\{\tau_n>T-t\}}\right]\\
&+e^{-r(T-t)}\E\left[\sup_{0\le u\le |t-t_n|}\left|e^{-r u}G(X^{z}_{T-t+u})\!-\!G(X^{z}_{T-t})\right|\right].
\end{align*}
By dominated convergence and right-continuity of $t\mapsto X_t$ we obtain
\[
\lim_{n\to\infty}\E\left[\sup_{0\le u\le |t-t_n|}\left|e^{-r u}G(X^{z}_{T-t+u})\!-\!G(X^{z}_{T-t})\right|\right]=0.
\]
For the remaining term, we notice that $\{\tau_n>T-t\}=\{\tau_n>T-t\}\cap \{\tau_n\le T-t_n\}\subset \{\tau_n\le T\}$ for all $n\ge 1$. Hence
\begin{align*}
\E\left[e^{-r\tau_n}\left(G(X^{z_n}_{\tau_n})\!-\!G(X^{z}_{\tau_n})\right)\I_{\{\tau_n>T-t\}}\right]\le& \E\left[e^{-r\tau_n}\left|G(X^{z_n}_{\tau_n})\!-\!G(X^{z}_{\tau_n})\right|\I_{\{\tau_n\le T\}}\right]\\
\le &\eta+\|G\|_{\cA_d}c_2(1+|x|_d)\big(\sqrt{\delta}+e^{rT/2}/\sqrt{m}\big),
\end{align*}
by the same arguments as in step 1 but with $\{\tau_n\le T\}$ instead of $\{\tau_n\le S\}$, for some constant $c>0$ and any $\delta,\eta>0$ and $m\ge 1$.

Letting $\delta,\eta\downarrow 0$ and $m\to\infty$ we conclude.
\end{proof}

We are now ready to prove Theorem \ref{thm:prob1}.
\begin{proof}{{\bf(Proof of Theorem \ref{thm:prob1})}}
We only need to show that $\Gamma$ is a contraction in $\cA_d$. Optimality of $\tau_*$ and the (super)martingale property of the value function $v$ will then follow from Lemma \ref{lem:OS}, upon choosing $f=v$ in all statements. We only give the proof for the infinite-time horizon as the one for the finite-time horizon is identical up to a change of notation.

First we prove that $\Gamma$ maps $\cA_d$ into itself. Fix $f\in\cA_d$ and recall that, by Lemma \ref{lem:OS} and Lemma \ref{lem:usc-G}, the mapping $x\mapsto (\Gamma f)(x)$ is continuous from $\R^d$ to $\R_+$. Then, since $\varphi\in\cA_d$, using Lemma \ref{lem:Lambda} and Cauchy-Schwarz inequality, for any $\tau\in\cT$, we obtain
\begin{align}\label{eq:cG0}
\left|\E\left[e^{-r\tau}\max\{\varphi(X^x_\tau),(\Lambda f)(X^x_\tau)\}\right]\right|^2\le c_0\E\left[e^{-2r\tau}(1+|X^x_\tau|^2_d)\right]\le c_0 (1+|x|^2_d),
\end{align}
where the final inequality follows from (i) in Assumption \ref{ass:sub} and the positive constant $c_0$ depends on $\|\varphi\|_{\cA_d}$ and $\|(\Lambda f)\|_{\cA_d}$. Using \eqref{eq:cG0} it is immediate to see that $\|(\Gamma f)\|_{\cA_d}\le \sqrt{c_0}$, hence $\Gamma f\in\cA_d$.

To prove that $\Gamma$ is a contraction, take $f\in\cA_d$ and $g\in\cA_d$ and denote by $\tau_*^f$ and $\tau^g_*$ the optimal stopping times as in \eqref{eq:tau} for $\Gamma f$ and $\Gamma g$, respectively. Fix $x\in\R^d$, then
\begin{align}\label{eq:cG1}
(\Gamma f)(x)-(\Gamma g)(x) \le&\E\left[e^{-r\tau_*^f}|(\Lambda f)(X^x_{\tau_*^f})-(\Lambda g)(X^x_{\tau_*^f})|\right] \\
\le & \E\left[e^{-2r\tau_*^f}|(\Lambda f)(X^x_{\tau_*^f})-(\Lambda g)(X^x_{\tau_*^f})|^2\right]^{\frac{1}{2}}\\
\le & \E\left[e^{-2r\tau_*^f}\frac{|(\Lambda f)(X^x_{\tau_*^f})-(\Lambda g)(X^x_{\tau_*^f})|^2}{1+|X^x_{\tau_*^f}|_d^2}(1+|X^x_{\tau_*^f}|_d^2)\right]^{\frac{1}{2}}\\
\le &\|(\Lambda f)-(\Lambda g)\|_{\cA_d}\E\left[e^{-2r\tau_*^f}(1+|X^x_{\tau_*^f}|_d^2)\right]^{\frac{1}{2}}\\
\le &\|(\Lambda f)-(\Lambda g)\|_{\cA_d}(1+|x|^2_d)^{\frac{1}{2}},
\end{align}
where in the first inequality we use that $\tau_*^f$ is sub-optimal for $(\Gamma g)(x)$ and $z\mapsto \max\{\varphi,z\}$ is $1$-Lipschitz, in the second one we use Jensen's inequality and in the final one we use (i) in Assumption \ref{ass:sub}.

Using the same argument, with $\tau^g_*$ in place of $\tau^f_*$ we also obtain
\begin{align}\label{eq:cG2}
(\Gamma g)(x)-(\Gamma f)(x)\le \|(\Lambda f)-(\Lambda g)\|_{\cA_d}(1+|x|^2_d)^{\frac{1}{2}}
\end{align}
and therefore, combining \eqref{eq:cG1} and \eqref{eq:cG2}, we get
\begin{align}\label{eq:estimate}
\frac{|(\Gamma f)(x)-(\Gamma g)(x)|}{(1+|x|^2_d)^{1/2}}\le \|(\Lambda f)-(\Lambda g)\|_{\cA_d}.
\end{align}
Taking the supremum over $x\in \R^d$ in \eqref{eq:estimate} leads to
\begin{align}\label{eq:cG3}
\|(\Gamma f)-(\Gamma g)\|_{\cA_d}\le \|(\Lambda f)-(\Lambda g)\|_{\cA_d}.
\end{align}

Moreover, for every fixed $x\in\R^d$, using triangular inequality and Jensen's inequality we get
\begin{align}\label{eq:cG4}
&\left|(\Lambda f)(x)-(\Lambda g)(x)\right|\\
&\le (1-p)\int_0^\infty e^{-rt} \E[\left|f(X^x_t)-g(X^x_t)\right|]F(\ud t)\\
&= (1-p)\int_0^\infty e^{-rt} \E\left[\frac{\left|f(X^x_t)-g(X^x_t)\right|}{(1+|X^x_t|_d^2)^{1/2}}
(1+|X^x_t|_d^2)^{1/2}\right]F(\ud t)\\
&\le (1-p)\|f-g\|_{\cA_d}\int_0^\infty \left(\E\left[ e^{-2rt}(1+|X^x_t|_d^2)\right]\right)^{1/2}F(\ud t)\\
&\le (1-p)\|f-g\|_{\cA_d}(1+|x|_d^2)^{1/2},
\end{align}
where the last inequality uses (i) in Assumption \ref{ass:sub}. From \eqref{eq:cG4} we deduce $\|(\Lambda f)-(\Lambda g)\|_{\cA_d}\le (1-p)\|f-g\|_{\cA_d}$ which, plugged back into \eqref{eq:cG3}, gives
\[
\|(\Gamma f)-(\Gamma g)\|_{\cA_d}\le (1-p)\|f-g\|_{\cA_d}.
\]
Since $p\in(0,1)$, the operator $\Gamma$ is a contraction and the proof is complete.
\end{proof}

The arguments of proof employed above require no assumption on the cumulative distribution function $F$. However, there is one particular case which deserves a comment. Intuitively, if the payoff $\psi(X)$ is revealed with no delay, i.e.~$\P(\vartheta=0)=1$, the optimiser would always choose $\alpha=1$ in \eqref{eq:value}. Indeed, if $\psi(X)$ is not achieved on the first attempt (i.e., with probability $1-p$) the investor learns about it immediately and she will instantly stop again and choose $\alpha=1$. Formally, this mechanism continues (instantaneously) until the payoff is attained. Then our problem reduces to a standard stopping problem with gain function $\psi$. These heuristics are confirmed in the next corollary.
\begin{corollary}\label{cor:F}
If $F(0)=1$ we have
\begin{align}\label{eq:vpsi}
v(x)=\sup_{\tau\in\cT}\E_x\left[e^{-r\tau}\psi(X_\tau)\right], \quad \text{for $x\in\R^d$}.
\end{align}
\end{corollary}
\begin{proof}
From Theorem \ref{thm:prob1} we know that $v$ is well defined and $v\ge \varphi$. Then by using that $F(0)=1$ and $\psi\ge \varphi$ we have
\begin{align}\label{eq:psi}
\max\{\varphi(x),(\Lambda v)(x)\}=\max\{\varphi(x),p\psi(x)+(1-p)v(x)\}=p\psi(x)+(1-p)v(x).
\end{align}
 Using \eqref{eq:psi} we get
\begin{align}\label{eq:vF}
v(x)=\sup_{\tau\in\cT}\E_x\left[e^{-r\tau}\left(p\psi(X_\tau)+(1-p)v(X_\tau)\right)\right]
\end{align}
and choosing $\tau=0$ we also obtain $v(x)\ge p\psi(x)+(1-p)v(x)$. Therefore $v\ge \psi$ and \eqref{eq:vF} gives
\begin{align}\label{eq:vge}
v(x)\ge \sup_{\tau\in\cT}\E_x\left[e^{-r\tau}\psi(X_\tau)\right].
\end{align}

For the reverse inequality we recall that $t\mapsto e^{-rt}v(X_t)$ is a $\P_x$-supermartingale (Theorem \ref{thm:prob1}), so that
\begin{align}
v(x)\le \sup_{\tau\in\cT}\E_x\left[e^{-r\tau}p\psi(X_\tau)\right]+(1-p)v(x).
\end{align}
Rearranging terms in the expression above and combining it with \eqref{eq:vge} leads to \eqref{eq:vpsi}.
\end{proof}

\begin{remark}\label{rem:flow}
It is worth noticing that the proof of Lemma \ref{lem:usc-G} does not use continuity of the flow $x\mapsto X^x_t$. We could have used the same arguments to prove that $\Lambda f\in C(\R^d)$ in Lemma \ref{lem:Lambda} and that $\Gamma f$ is lower semi-continuous in Lemma \ref{lem:OS}. Hence, Theorem \ref{thm:prob1} and Corollary \ref{cor:F} hold without the assumption of continuity of the flow $x\mapsto X^x_t$.
\end{remark}

\begin{remark}[Variational inequality]\label{rem:VI}
By continuity of the flow $x\mapsto X^x_t$, the process $X$ is a Feller process. Denoting its infinitesimal generator by $\cL$ we may formally expect the value function $v$ to be solution (in a suitable sense) of the variational inequality
\begin{equation}\label{eq:VI}
\max\big\{(\cL v- r v)(x),\max\{\varphi,\Lambda v\}(x)-v(x)\big\}=0, \quad x\in\R^d.
\end{equation}
In the infinite time-horizon problem we also need to add linear growth conditions at infinity (as we expect $v\in\cA_d$), whereas in the finite time-horizon problem we have the terminal condition \eqref{eq:vT}. In general, existence and regularity of a solution to the variational inequality above depend on the structure of the operator $\cL$. The problem is also challenging due to the non-local (recursive) nature of the operator $\Lambda$ (similar technical difficulties arise in HJB equations related to impulse control problems).

In Section \ref{sec:fbp} we show for a specific problem that indeed $v$ solves the variational inequality, written in the form of a free boundary problem.
\end{remark}

\section{Application to stock trading with the dark pool}\label{sec:appl}

In this section we discuss the application of the recursive optimal stopping problem to the problem of trading in different venues, introduced in Section \ref{sec:examples}. We consider a trader who wants to sell a certain number of shares of a stock, in a single transaction. At any (stopping) time the trader may decide to sell the whole inventory in the traditional market exchange or in a dark pool. Since we do not allow for partial execution, with no loss of generality we will later assume that the inventory consists of a single share\footnote{Equivalently, one may consider a discrete list of small orders to be liquidated according to a fixed sequence of transactions. This is a common  procedure among large traders, which is normally applied in order to reduce price impact arising from large selling orders, that typically push prices down (see, e.g. \cite[Ch. 6 and Ch. 7]{cartea2015algorithmic}).}.

The execution of orders in the two markets obeys different mechanisms and the sale prices are also different. In the standard exchange the order is certainly executed instantaneously, whereas in the dark pool orders are executed if a matching order arrives, i.e., only with some probability $p\in(0,1)$ and with a delay that may vary across different orders. This means that after a certain random time $\vartheta$, with probability $1-p$ the order is either not executed or cancelled by the trader.

We denote by $S=(S_t)_{t \geq 0}$ the (non-negative) bid price process. Sales in the standard exchange are subject to price impact and, in order to account for this feature, we say that the sale price of the stock in this market, at time $\tau$, is $\gamma S_\tau$ for some constant $\gamma\in (0,1]$. Since the trader is interested in a single sale for a fixed number of shares, the use of a fixed (proportional) price impact (given by $\gamma$) seems a reasonable choice that leads to a tractable model.

In the dark pool the stock can be sold at a more favourable price (typically the mid price between bid and ask) with no price impact.
Hence, we let $K=(K_t)_{t \geq 0}$ be a non-negative process representing a spread on the bid price. If an order placed in the dark pool at time $\tau$ is executed, the trader receives $S_{\tau+\vartheta}+K_{\tau+\vartheta}$ at time $\tau+\vartheta$. Alternatively, if the order {\em is not} executed (or cancelled) the trader must start her optimization afresh. The investor is therefore committed to dark pools for the entire random waiting time (Remark \ref{rem:theta}).

\subsection{Setting and reduction to one dimension}\label{sec:setting-ex}
Let $(\Omega, \cF, \P)$ be a probability space and consider two independent Brownian motions $(B^1_t)_{t\ge 0}$, $(B^2_t)_{t\ge 0}$. Let $\bF$ be the natural filtration generated by $B^1$ and $B^2$, completed with $\P$-null sets. We model the price process $S$ and the spread $K$ by correlated diffusions as follows:

\begin{align}
\label{eq:bid}&\ud S_t=\mu_1 S_t \ud t +\sigma_1 S_t \ud B^1_t,\:\:\:\qquad S_0=s>0,\\
\label{eq:spread}&\ud K_t=\mu_2K_t\ud t +\sigma_2 K_t (\nu \ud B^1_t+\sqrt{1-\nu^2} \ud B^2_t),\qquad K_0=k>0.
\end{align}
where $\mu_1, \mu_2\in\R$ and $\sigma_1,\sigma_2>0$ are constants and $\nu\in[-1,1]$.

The problem formulation corresponds to that of Section \ref{sec:infinite} where $X=(S,K)$, $\varphi(X)=\gamma S$, $\psi(X)=S+K$ and $\cT$ is unbounded. So equation \eqref{eq:value}, and its equivalent formulation given in equation \eqref{eq:vtilde}, read as
\begin{align}\label{eq:value2}
v(s,k)=&\sup_{(\tau, \alpha)\in \D}\E_{s,k}\bigg[e^{-r \tau} \gamma S_\tau \I_{\{\alpha=0\}}\! \\
&\qquad\qquad+ \!e^{-r (\tau+\vartheta) } \!  \left( p (S_{\tau+\vartheta}\!+\!K_{\tau+\vartheta})\!+\! (1-p) v(S_{\tau+\vartheta},K_{\tau+\vartheta})\right) \I_{\{\alpha=1\}} \bigg]\\
=&\sup_{\tau\in \cT}\E_{s,k}\left[e^{-r \tau} \max\left\{\gamma S_\tau,(\Lambda v)(S_\tau,K_\tau) \right\}\right].
\end{align}
In this setting, for any continuous function $f:\R^2_+\to\R_+$ we have
\begin{align}\label{eq:Lambdaf}
(\Lambda f)(s,k)=\int_0^\infty e^{-rt}\E_{s,k}\left[p(S_t+K_t)+(1-p)f(S_t,K_t)\right]F(\ud t)
\end{align}
and the second equality in \eqref{eq:value2} holds because of Lemma \ref{lem:value}.

Note that, in this example, the processes $S$ and $K$ are positive and our state space is $\R^2_+:=(0,\infty)^2$. Then, instead of working on the Banach space $\mathcal A_2$,  we can consider the space $\cA^+_2$ defined as in \eqref{eq:A} but with $\R^2_+$ in place of $\R^d$, i.e.
\begin{align}\label{eq:A2}
\cA^+_2:=\left\{
f: f\in C(\R^2_+;\R_+),\:\text{such that}\:\|f\|_{\cA^+_2}<\infty
\right\}.
\end{align}
with $\displaystyle \|f\|^2_{\cA_2^+}:=\sup_{x\in\R^2_+}\frac{|f(x)|^2}{1+|x|_2^2}.$

\begin{remark}\label{rem:spaceAprime}
Let $\tilde{r}=r(1-\rho)$. The process $\hat{X}=(e^{-2\tilde{r}}(1+S^2_t+K^2_t))_{t \ge 0}$ is a supermartingale if its It\^o dynamics contains a negative drift; that is, if
\[
-2\tilde{r}(1+S^2_t+K^2_t) +S^2_t(2\mu_1+\sigma_1^2) +K^2_t(2\mu_2+\sigma_2^2)\leq 0.
\]
Hence, a sufficient condition for $\hat{X}$ to be a supermartingale is $\tilde{r}\ge\mu_i+\frac{1}{2}\sigma_i^2$, $i=1,2$ and if $r>\mu_i+\frac{1}{2}\sigma_i^2$, $i=1,2$, we can find $\rho\in(0,1)$ such that (i) in Assumption \ref{ass:sub} holds. Moreover, this also guarantees that (ii) and (iii) of Assumption \ref{ass:sub} are fulfilled.

Notice that, due to the explicit form of the processes involved, one could repeat arguments as in Section \ref{sec:value} to prove that a fixed point can be found in the space
\begin{align}\label{eq:A2b}
\cA'_2:=\left\{f: f\in C(\R_+^2;\R_+),\:\text{such that}\:\|f\|_{\cA'_2}<\infty\right\},
\end{align}
with the norm
\[
\|f\|_{\cA'_2}:=\sup_{(s,k)\in\R_+^2}\frac{|f(s,k)|}{(1+s+k)},
\]
under weaker conditions than those in Assumption \ref{ass:sub}. In particular, it is sufficient to replace (i) of Assumption \ref{ass:sub} by the condition: $e^{-rt}(1+S^s_t+K^k_t)$ is a supermartingale. We will discuss this alternative approach in more detail in Section \ref{sec:comparison} and Appendix \ref{sec:appGBM}.
\end{remark}

In light of the above remark, and in order to avoid repetitions, here we simply assume that $r>\mu_i+\sigma^2_i/2$ for $i=1,2$ so that all results from Section \ref{sec:value} apply to the current setting. Moreover, with no loss of generality we take $\gamma=1$ in \eqref{eq:value2}, for notational simplicity. It will be clear that all results below also hold for any other $\gamma\in(0,1)$.

The problem stated in \eqref{eq:value2} has some interesting features. The first one is that the value function is homogeneous in $s$, as shown in the next lemma.
\begin{lemma}\label{lem:homo}
For all $(s,k)\in\R^2_+$ we have $v(s,k)=s\, v(1,k/s)$.
\end{lemma}
\begin{proof}
Since $v$ is the unique fixed point of the operator $\Gamma$ defined in \eqref{eq:gamma}, for any $f_0\in\cA^+_2$, setting $f_{n+1}=(\Gamma f_n)$ for $n\ge 0$, we have
\begin{align}\label{eq:Gfn}
v=\lim_{n\to\infty}(\Gamma f_n),
\end{align}
where the limit is taken in $\cA^+_2$. Therefore, homogeneity of $v$ in the $s$ variable holds if such property is satisfied by $f_n$, for every $n \in \mathbb{N}$.

We proceed by induction and assume that $f_n$ is homogeneous in $s$, i.e.~$f_n(s,k)=s f_n(1,k/s)$. Since $S^s_t=s\,S^1_t$ and $K^{k/s}_t=s^{-1}K^k_t$ we obtain
\begin{align}\label{eq:homo0}
(\Lambda f_n)(s,k)\!=\!s\! \int_0^\infty\!\!\! e^{-rt}\E\left[p\left(S^1_t\!+\!K^{k/s}_t\right)\!+\!(1\!-\!p)f_n\left(S^1,K^{k/s}_t\right)\right]\!F(\ud t)=s(\Lambda f_n)(1,k/s).
\end{align}
Therefore
\begin{align}\label{eq:fn+1}
f_{n+1}(s,k)=&(\Gamma f_n)(s,k)=\sup_{\tau\in \cT}\E\left[e^{-r\tau}\max\{S^s_{\tau}, (\Lambda f_n)(S^s_\tau, K^k_\tau)\}\right]\\
=&\sup_{\tau\in \cT}\E\left[e^{-r\tau}\max\{s S^1_{\tau}, s(\Lambda f_n)(S^1_\tau, K^{k/s}_\tau)\}\right]=sf_{n+1}(1,k/s).
\end{align}
Hence, $f_{n+1}$ is also homogeneous in the $s$ variable, which concludes the proof thanks to \eqref{eq:Gfn}.
\end{proof}

In the next proposition, we use Lemma \ref{lem:homo} and the dynamics of $S$ and $K$ (see \eqref{eq:bid}--\eqref{eq:spread}) to reduce the dimension of the state space. For this let $\beta_1:=\sigma_2\nu-\sigma_1$ and $\beta_2:=\sigma_2\sqrt{1-\nu^2}$ and consider a process $Z$ defined as the unique strong solution of
\begin{align}\label{eq:Z3}
\ud Z^z_t = (\mu_2-\mu_1) {Z^z_t} \ud t+ \sqrt{\beta^2_1+\beta^2_2} {Z^z_t} \, \ud \tilde B_t,
\end{align}
with initial condition $Z^z_0=z>0$, where $\tilde B:=(\tilde B_t)_{t\ge 0}$ is the $\P$-Brownian motion given by
\[
\widetilde B_t=\frac{\beta_1 B^1_t}{\sqrt{\beta_1^2+\beta_2^2}}+\frac{\beta_2 B^2_t}{\sqrt{\beta_1^2+\beta_2^2}}\quad\text{for $t \geq 0$.}
\]

Then we also introduce the operator
\begin{align}\label{eq:Pi}
(\Pi g)(z):= \int_0^\infty e^{-(r-\mu_1)t}\E_z\left[p(1+Z_t)+(1-p)g(Z_t)\right]F(\ud t),
\end{align}
for any $g\in \cA^+_1$, where $\cA^+_1$ is defined as in \eqref{eq:A2} but replacing $\R^2_+$ by $\R_+$. The operator $\Pi$  plays the role of the operator $\Lambda$ from \eqref{eq:Lambda} but in the one dimensional setting.

Similarly to  \eqref{eq:gamma}, for any $g\in\cA^+_1$ we also define the operator $\tilde \Gamma$
\begin{align}\label{eq:ug}
(\tilde\Gamma g)(z):=\sup_{\tau\in \cT}\E_z\left[e^{-(r-\mu_1)\tau}\max\{1,(\Pi g)(Z_\tau)\}\right].
\end{align}
Since $r>\mu_1+\sigma^2_1/2$, it would not be difficult to adapt the proofs from the previous sections to show that $\tilde\Gamma$ admits a unique fixed point in $\cA^+_1$. However, we follow a slightly different line of arguments.

In Proposition \ref{prop:dim} below we formulate an optimal stopping problem equivalent to \eqref{eq:value2} in the reduced state space. The proof requires the following preliminary lemma.
\begin{lemma}\label{lem:Girs}
Fix a deterministic $T>0$ and define a probability measure $\Q$ on $\cF_T$ with density
\begin{align}\label{eq:Girs}
\frac{\ud \Q}{\ud \P}\Big|_{\cF_T}:=e^{\sigma_1 B^1_T-\frac{\sigma_1^2}{2}T}=S_Te^{-\mu_1T}.
\end{align}
Let $\hat Z$ be defined as $\hat Z_t:=K_t/S_t$, for $t\ge 0$. Then, recalling $Z$ in equation \eqref{eq:Z3}, we have
\begin{align}\label{laws}
\mathsf{Law}\left((\hat Z_t)_{t\in[0,T]}\big|\Q\right)=\mathsf{Law}\left((Z_t)_{t\in[0,T]}\big|\P\right).
\end{align}
\end{lemma}
\begin{proof}
The measure $\Q$ in \eqref{eq:Girs} is equivalent to $\P$ on $\cF_T$ and by Girsanov Theorem $B^\Q _t:=B^1_t-\sigma_1 t$ is a $\Q$-Brownian motion for every $t\in[0,T]$.
By applying It\^o formula to $\hat Z_t:=K_t/S_t$ we get that the dynamics of $\hat Z$ under $\Q$ is
\begin{align}\label{eq:Z2}
\frac{\ud \hat Z^z_t}{\hat Z^z_t} =(\mu_2-\mu_1)\ud t+ \beta_1 \ud B^\Q_t +\beta_2 \ud B^2_t=(\mu_2-\mu_1)\ud t+ \sqrt{\beta_1^2+\beta_2^2} \ud \widetilde B^\Q_t
\end{align}
where $\beta_1$ and $\beta_2$ are as in \eqref{eq:Z3} and $\widetilde B^\Q$ is the $\Q$-Brownian motion given by
\[
\widetilde B^\Q_t=\frac{\beta_1 B^\Q_t}{\sqrt{\beta_1^2+\beta_2^2}}+\frac{\beta_2 B^2_t}{\sqrt{\beta_1^2+\beta_2^2}},\quad\text{for $t \in [0,T]$}.
\]
Comparing \eqref{eq:Z2} to equation \eqref{eq:Z3}, it is clear that $\hat Z$ under $\Q$ has the same law as $Z$ under $\P$, which concludes the proof.
\end{proof}

\begin{proposition}\label{prop:dim}
For $(k,s)\in\R_+^2$ let $z=k/s$ and set $u(z):= v(1,k/s)$. Then, $u\in\cA^+_1$ and it is the unique solution to $u=(\tilde \Gamma u)$.
\end{proposition}
\begin{proof}
We first observe that $v\in\cA^+_2$ and Assumption \ref{ass:sub} imply $u\in\cA^+_1$. Then, we need to show that $u$ is the unique solution of $u=(\tilde \Gamma u)$. The idea is to use a change of measure argument but we need some care, due to possibly infinite stopping times.

For any fixed $T>0$, the law of $\hat Z$ under $\Q$ is the same as the law of $Z$ under $\P$ on the interval $[0,T]$, by  Lemma  \ref{lem:Girs}. Then,
thanks to Lemma \ref{lem:homo} and using \eqref{eq:Girs} and the explicit solution of \eqref{eq:bid}--\eqref{eq:spread}, for each $T>0$ we have
\begin{align}\label{eq:u00}
(\Lambda^T v)(s,k):=&\int_0^T e^{-rt}\E_{s,k}\left[p(S_t+K_t)+(1-p)v(S_t,K_t)\right]F(\ud t)\\
=&\int_0^T e^{-rt}\E\left[S^s_tp(1+\hat Z^z_t)+(1-p)S^s_tv(1,\hat Z^z_t)\right]F(\ud t)\\
\!=&\,s\! \int_0^T\!\!\! e^{-(r-\mu_1)t}\E^\Q\left[p(1\!+\!\hat Z^{z}_t)\!+\!(1\!-\!p)v(1,\hat Z^{z}_t)\right]\!F(\ud t)\\
\!=&\,s\! \int_0^T\!\!\! e^{-(r-\mu_1)t}\E\left[p\left(1\!+\! Z^{z}_t\right)\!+\!(1\!-\!p)u\left(Z^{z}_t\right)\right]\!F(\ud t),
\end{align}
where in the final equation we used \eqref{laws} and that $u(z)=v(1,z)$ by definition. Then, we have $(\Lambda^T v)(s,k)=s(\Pi^T u)(z)$ with
\begin{align}
(\Pi^T u)(z):= \int_0^T\!\!\! e^{-(r-\mu_1)t}\E\left[p\left(1\!+\! Z^{z}_t\right)\!+\!(1\!-\!p)u(Z^{z}_t)\right]\!F(\ud t).
\end{align}
Now, taking limits as $T\to\infty$ we obtain
\begin{align}\label{eq:u0}
(\Lambda v)(s,k)=\lim_{T\to\infty}(\Lambda^T v)(s,k)=s\lim_{T\to\infty}(\Pi^T u)(z)=s(\Pi u)(z).
\end{align}

Plugging \eqref{eq:u0} into \eqref{eq:value2} we get
\begin{align}\label{eq:u10}
v(s,k)=&\sup_{\tau\in \mathcal T}\E\left[e^{-r\tau}\max\{S^s_\tau,S^s_\tau (\Pi u)(\hat Z^z_\tau)\}\right].
\end{align}
For each $T>0$ we define
\begin{align}\label{eq:u11}
v^T(s,k):=&\sup_{\tau\in \cT}\E\left[e^{-r(\tau\wedge T)}\max\{S^s_{\tau\wedge T},S^s_{\tau\wedge T} (\Pi u)(\hat Z^z_{\tau\wedge T})\}\right]\\
=& s\sup_{\tau\in \cT}\E^\Q\left[e^{-(r-\mu_1)(\tau\wedge T)}\max\{1,(\Pi u)(\hat Z^z_{\tau\wedge T})\}\right]\notag\\
=& s\sup_{\tau\in \cT}\E\left[e^{-(r-\mu_1)(\tau\wedge T)}\max\{1,(\Pi u)(Z^z_{\tau\wedge T})\}\right],\notag
\end{align}
where the first equality comes from \eqref{eq:u0}, the second one from the change of measure and the final one from \eqref{laws}. Recalling \eqref{eq:ug}, it is natural to set
\begin{align}
(\tilde\Gamma^T u)(z):=\sup_{\tau\in \cT}\E\left[e^{-(r-\mu_1)(\tau\wedge T)}\max\{1,(\Pi u)(Z^z_{\tau\wedge T})\}\right],
\end{align}
so that \eqref{eq:u11} reads
\begin{align}\label{eq:vTGT}
v^T(s,k)=s\,(\tilde\Gamma^T u)(z).
\end{align}

Next, we want to prove that
\begin{align}\label{eq:limT}
\lim_{T\to\infty}v^T(s,k)=v(s,k)\quad\text{and}\quad\lim_{T\to\infty}(\tilde\Gamma^T u)(z)=(\tilde\Gamma u)(z).
\end{align}
We give the full argument of \eqref{eq:limT} for $(\tilde\Gamma^T u)$ as the computations for $v^T$ are analogous.

First, $(\tilde\Gamma^T u)\le(\tilde\Gamma u)$ on $\R_+$ since stopping times in \eqref{eq:u11} are bounded by $T$.  Second, $T\mapsto (\tilde\Gamma^T u)$ is increasing as the set of admissible times increases.
Then
\begin{align}\label{limup}
\lim_{T\to\infty}(\tilde\Gamma^T u)(z)\le(\tilde\Gamma u)(z),\qquad \text{for $z\in\R_+$}.
\end{align}
For the reverse inequality we notice that, for any stopping time $\tau$, Fatou's lemma and continuity of the gain process give
\begin{align}\label{limdown}
&\E\left[e^{-(r-\mu_1)\tau}\max\{1,(\Pi u)(Z^z_{\tau})\}\right]\\
&\le \liminf_{T\to\infty}\E\left[e^{-(r-\mu_1)(\tau\wedge T)}\max\{1,(\Pi u)(Z^z_{\tau\wedge T})\}\right]\le \liminf_{T\to\infty}(\tilde\Gamma^T u)(z),\qquad \text{for $z>0$}.
\end{align}
Hence, \eqref{limup} and \eqref{limdown} imply \eqref{eq:limT}. Taking limits in \eqref{eq:vTGT} and using \eqref{eq:limT} we obtain
\begin{align}\label{eq:u1}
v(s,k)=s\, (\tilde\Gamma u)(z).
\end{align}

Finally, from \eqref{eq:u1} and Lemma \ref{lem:homo} we obtain
\[
(\tilde\Gamma u)(z)=s^{-1}v(s,k)=v(1,k/s)=u(z).
\]
Uniqueness of the fixed point for $u$ follows from uniqueness of the fixed point for $v$.
\end{proof}

Thanks to Proposition \ref{prop:dim}, we know that the recursive stopping problem
\begin{align}\label{eq:u}
u(z)=\sup_{\tau\in \cT}\E_z\left[e^{-(r-\mu_1)\tau}\max\{1,(\Pi u)(Z_\tau)\}\right],\quad z\in\R_+,
\end{align}
is well-posed and, recalling also Theorem \ref{thm:prob1}, we obtain a simple corollary.

\begin{corollary}\label{cor:OS-u}
The stopping time
\begin{align}\label{eq:tauhat}
\hat \tau:=\inf\left\{t\ge 0\,:\, u(Z_t)=\max[1,(\Pi u)(Z_t)]\right\}
\end{align}
is optimal for \eqref{eq:u}. Moreover, the process
\[
t\mapsto e^{-(r-\mu_1)t}u(Z_t),\qquad t\in[0,\infty]
\]
is a continuous (non-negative) $\P_z$-supermartingale and the process
\[
t\mapsto e^{-(r-\mu_1)(t\wedge\hat \tau)}u(Z_{t\wedge\hat\tau}),\qquad t\in[0,\infty)
\]
is a continuous (non-negative) $\P_z$-martingale, for any $z\in\R_+$.
\end{corollary}

Let us choose $f_0\in\cA^+_2$ such that $f_0(s,k)=s f_0(1,k/s)$. For $n\ge 0$, set $f_{n+1}=(\Gamma f_n)$ and $g_n(z):=f_n(1,z)$ for $z\in\R_+$. We can easily check that $g_n\in\cA^+_1$ since $f_n\in\cA^+_2$. Moreover, \eqref{eq:fn+1} implies that $f_{n+1}(s,k)=s f_{n+1}(1,k/s)=sg_{n+1}(z)$. Hence,
repeating the argument of proof of Proposition \ref{prop:dim}, we obtain
\begin{align}\label{eq:un0}
sg_{n+1}(z)=f_{n+1}(s,k)=(\Gamma f_{n})(s,k)=s \sup_{\tau\in \cT}\E\left[e^{-(r-\mu_1)\tau}\max\{1,(\Pi g_n)(Z^z_\tau)\}\right].
\end{align}

The next is a simple corollary of \eqref{eq:un0} and of the fact that $\Gamma$ (and hence $\tilde \Gamma$) is a contraction.
\begin{corollary}\label{cor:u}
Let $g_0\in\cA_1^+$ be arbitrary and define $g_{n+1}:=\tilde\Gamma g_n$ for $n\ge 0$. Then
\[
u=\lim_{n\to\infty}\tilde\Gamma g_n,
\]
where the limit is taken in $\cA^+_1$.
\end{corollary}

\begin{remark}\label{rem:properties}{\bf(Properties of $Z$, $u$ and $\Pi u$).}
\begin{itemize}
\item[(a)] (Asymptotic growth). Recalling that $r>\mu_1\vee\mu_2$ and using the explicit form of the solution of \eqref{eq:Z3} we have
\begin{align}\label{eq:limsupZ}
\limsup_{t\to\infty} e^{-(r-\mu_1)t}Z_t=0,\qquad\text{$\P_z$-a.s., for all $z\in\R_+$}.
\end{align}
Then, combining \eqref{eq:limsupZ} with the fact that $u\in\cA^+_1$ (i.e., $u$ has sublinear growth) we get
\begin{align}\label{eq:limsupu}
\limsup_{t\to\infty}e^{-(r-\mu_1)t}u(Z_t)=0,\qquad\text{$\P_z$-a.s., for all $z\in\R_+$.}
\end{align}

\item[(b)] (Supermartingale property). From \eqref{eq:Pi}, using Fubini's theorem and the strong Markov property we have
\begin{align}
\E\left[e^{-(r-\mu_1)\tau}(\Pi u)(Z^z_\tau)\right]\!=\!\int_0^\infty\!\!\E \left[e^{-(r-\mu_1)(\tau+t)}\left(p(1\!+\!Z^z_{\tau+t})\!+\!(1\!-\!p)u(Z^z_{\tau+t})\right)\right] F(\ud t),
\end{align}
for any stopping time $\tau\in\cT$. Now, $t\mapsto e^{-(r-\mu_1)t}(1+Z_t)$ and $t\mapsto e^{-(r-\mu_1)t}u(Z_t)$ are non-negative and uniformly integrable supermartingales by Assumption \ref{ass:sub} and Corollary \ref{cor:OS-u}. Hence, they are supermartingales on $[0,\infty]$. Moreover, for $s\ge t$ we have $\{\tau+t\le s\}\in\cF_{s-t}\subseteq \cF_{s}$, so that $\tau+t$ is a stopping time in $\cT$. Then the optional sampling theorem gives
\begin{align}\label{eq:superm}
&\E\left[e^{-(r-\mu_1)\tau}(\Pi u)(Z^z_\tau)\right]\!\le\!\int_0^\infty\!\!\E\left[e^{-(r-\mu_1)t}\left(p(1\!+\!Z^z_{t})\!+\!(1\!-\!p)u(Z^z_{t})\right)\right] F(\ud t)\! =\! (\Pi u)(z),
\end{align}
for any $\tau\in\cT$.
\end{itemize}
\end{remark}

\subsection{Optimal boundaries and smooth-fit}\label{sec:boundaries}
In this section, we aim to study additional properties of the solution $u$ to the one dimensional problem \eqref{eq:u} that  will enable to characterize the optimal stopping rule (for both the one dimensional and the original two dimensional problem) in terms of two optimal boundaries. Moreover, we will prove that $u\in C^1(\R_+)$, hence $v\in C^1(\R^2_+)$.

The first result shows monotonicity and convexity of $u$.
\begin{proposition}\label{prop:u}
The function $u$ is monotonic non-decreasing and convex.
\end{proposition}
\begin{proof}
It follows from Corollary \ref{cor:u} that $u=\lim_{n\to\infty}\tilde\Gamma g_n$ in $\cA^+_1$. Thus, it is sufficient to show that if $g_n$ is non-decreasing and convex then $\tilde\Gamma g_n$ inherits such properties.

{\bf Step 1}. ({\em Monotonicity.}) Assume that $g_n$ is non-decreasing. Then by \eqref{eq:Z3} and \eqref{eq:Pi} we get that $\Pi g_n$ is also non-decreasing. This implies that $z\mapsto\max\{1,(\Pi g_n)(Z^z_\tau)\}$ is non-decreasing for any $\tau\in\cT$ given and fixed, and hence, by comparison arguments we have that $z\mapsto g_{n+1}(z)=(\tilde\Gamma g_{n})(z)$ is non-decreasing as well.

{\bf Step 2}. ({\em Convexity.}) Assume that $g_n$ is non-decreasing and convex. From \eqref{eq:Z3} and \eqref{eq:Pi} we immediately see that $\Pi g_n$ is convex too. Then, $z\mapsto\max\{1,(\Pi g_n)(z)\}$ is convex and non-decreasing. Let us now consider $z_1<z_2$ and $\lambda\in(0,1)$, and set $z_\lambda=\lambda z_1+(1-\lambda)z_2$. Using convexity of $\max\{1,(\Pi g_n)\}$, linearity of $z\mapsto Z^z_\tau$ (for $\tau\in\cT$ given and fixed) and the inequality $\sup(a+b)\le \sup(a)+\sup(b)$ we derive
\begin{align}
g_{n+1}(\lambda z_1+(1-\lambda)z_2)=&\sup_{\tau\in \cT}\E\left[e^{-(r-\mu_1)\tau}\max\left\{1,(\Pi g_n)(Z^{z_\lambda}_\tau)\right\}\right]\\
\le &\sup_{\tau\in \cT}\E\left[e^{-(r-\mu_1)\tau}\max\left\{1,\lambda(\Pi g_n)(Z^{z_1}_\tau)+(1-\lambda)(\Pi g_n)(Z^{z_2}_\tau)\right\}\right]\\
\le &\lambda\sup_{\tau\in \cT}\E\left[e^{-(r-\mu_1)\tau}\max\left\{1,(\Pi g_n)(Z^{z_1}_\tau)\right\}\right]\\
&+(1-\lambda)\sup_{\tau\in \cT}\E\left[e^{-(r-\mu_1)\tau}\max\left\{1,(\Pi g_n)(Z^{z_2}_\tau)\right\}\right]\\
=&\lambda g_{n+1}(z_1)+(1-\lambda)g_{n+1}(z_2).
\end{align}

Monotonicity and convexity of $u$ follow from the two steps above.
\end{proof}

The advantage of dealing with a convex function (of one variable) is that its first derivative has at most countably many points of discontinuity. We will now show that higher regularity holds for our value function.
\begin{proposition}\label{prop:C1}
We have that $u\in C^1(\R_+)$.
\end{proposition}
\begin{proof}
Assume, by contradiction, that there exists $\bar z\in\R_+$ such that
\begin{align}\label{eq:C1-1}
\bar c:=u'(\bar z+)-u'(\bar z-)>0,
\end{align}
where $u'(z\pm)$ are the right/left-derivatives of $u$ at a point $z$.
Denote $\zeta_\eps:=\inf\{t\ge0:Z^{\bar z}_t\notin(\bar z\!-\!\eps,\bar z\!+\!\eps)\}$, for $\eps>0$ given and fixed.
Then for any $t\in(0,1)$ we have
\begin{align}\label{eq:C1-0}
&\E\left[e^{-(r-\mu_1)(t\wedge\zeta_\eps)}u(Z^{\bar z}_{t\wedge\zeta_\eps})\right]\\
&\quad=u(\bar z)+\E\bigg[\int_0^{t\wedge\zeta_\eps} \!\!e^{-(r-\mu_1)s}\left(u'(Z^{\bar z}_s-)Z^{\bar z}_s(\mu_2\!-\!\mu_1)\!-\!(r\!-\!\mu_1)u(Z^{\bar z}_s)\right)\ud s\\
&\qquad\qquad\qquad+\!\tfrac{1}{2}\int_0^\infty\!\! L^a_{t\wedge\zeta_\eps}(Z^{\bar z}) u''(\ud a)\bigg]
\end{align}
thanks to It\^o-Tanaka-Meyer formula (see \citet[Thm.~70, Ch.~IV]{protter2005stochastic}), where $(L^a_t)_{t\ge 0}$ is the local time of the process $Z$ at a point $a\in\R_+$, $u''$ is understood as a non-negative measure, the left-derivative $u'(z-)$ is well defined by convexity and the martingale term has been removed.

We now notice that $u'$ is locally bounded since it is of bounded variation on $\R_+$ (Proposition \ref{prop:u}) and $u$ is bounded on $[\bar z-\eps,\bar z+\eps]$ by continuity. Then, using \eqref{eq:C1-1} and \eqref{eq:C1-0} we get
\begin{align}\label{eq:Lt}
&\E\left[e^{-(r-\mu_1)(t\wedge\zeta_\eps)}u(Z^{\bar z}_{t\wedge\zeta_\eps})\right]\ge u(\bar z)-c_{\eps}\E\left[t\wedge\zeta_\eps\right]+\!\tfrac{1}{2}\,\bar c\, \E\left[ L^{\bar z}_{t\wedge\zeta_\eps}(Z^{\bar z})\right],
\end{align}
where $c_\eps>0$ is a suitable constant independent of $t$. In the limit as $t\to0$ one has $\E[L^{\bar z}_{t\wedge\zeta_\eps}(Z^{\bar z})]\sim t^{1/2+\delta}$, for arbitrarily small $\delta>0$, and $\E\left[t\wedge\zeta_\eps\right]\sim t$ (see, e.g., eqs.~(34) and (35) in \cite{de2017optimal}). Hence the positive term in \eqref{eq:Lt} dominates and
\[
\E\left[e^{-(r-\mu_1)(t\wedge\zeta_\eps)}u(Z^{\bar z}_{t\wedge\zeta_\eps})\right]> u(\bar z)
\]
for sufficiently small $t\in(0,1)$. This inequality violates the supermartingale property of $u$ (see Corollary \ref{cor:OS-u}), thus implying that $\bar c=0$.
\end{proof}

Next, we will use properties of the value function $u$ to describe the geometry of the continuation and stopping region for the one dimensional problem \eqref{eq:u}. By monotonicity of $u$ (and of $\Pi u$) and noticing that $(\Pi u)(z)\uparrow \infty$ as $z\to\infty$, it is clear that there exists at most a unique point $z_0<\infty$ such that $(\Pi u)(z_0)=1$. To be more precise we set
\begin{align}\label{eq:z_0}
z_0:=\inf\{z\in\R_+\,:\,(\Pi u)(z)>1\}.
\end{align}
We want to show that $z_0>0$.

\begin{lemma}\label{lem:z_0}
We have $z_0>0$ if and only if $F(0)<1$.
\end{lemma}
\begin{proof}

We observe that since $(\Pi u)(z)$ is increasing and continuous, then $(\Pi u)(0)<1$ if an only if $z_0>0$.

{\bf Step 1}. ({\em$z_0>0\Rightarrow F(0)<1$}.) Assume $z_0>0$. Then $(\Pi u)(z)<1$ for $z\in[0,z_0)$. However, from Corollary \ref{cor:F} we know that if $F(0)=1$ it must be $(\Pi u)(z)=p(1+z)+(1-p)u(z)\ge 1$ for all $z\ge 0$ (see \eqref{eq:psi}), where the final inequality uses $u\ge 1$, by equation \eqref{eq:u}. Hence we reach a contradiction and $F(0)<1$.

{\bf Step 2}. ({\em$z_0>0\Leftarrow F(0)<1$}.) Let us assume $F(0)<1$ and let us prove $(\Pi u)(0)<1$.
Recall that $u=\lim_{n\to\infty}\tilde\Gamma g_n$ (see Corollary \ref{cor:u}). First, we show that if $g_n(0)=1$ then $z^n_0>0$, where
\[
z_0^n:=\inf\{z\in\R_+\,:\,(\Pi g_n)(z)>1\}.
\]
By dominated convergence, letting $z\downarrow 0$ and using that $Z^z_t\downarrow 0$, $\P$-a.s., for all $t\ge 0$ we obtain
\begin{align}\label{eq:Pi0}
(\Pi g_n)(0):=&\lim_{z\to 0}(\Pi g_n)(z)\\
=&[p\!+\!(1\!-\!p)g_n(0)]\!\int_0^\infty \!\!e^{-(r-\mu_1)t}F(\ud t)=\int_0^\infty \!\!e^{-(r-\mu_1)t}F(\ud t)<1,
\end{align}
where the final inequality uses $r>\mu_1$, $F(0)<1$ and $g_n(0)=1$. This establishes $z^n_0>0$.

Second, we show that  $g_{n+1}(0)=(\tilde\Gamma g_n)(0)=1$. Using again dominated convergence and that $Z^z_\tau\downarrow 0$ as $z\to 0$, $\P$-a.s., for any $\tau\in\cT$ given and fixed, we find
\begin{align}
(\tilde\Gamma g_n)(0):=&\lim_{z\to 0}(\tilde\Gamma g_n)(z)\\
=&\sup_{\tau\in \cT}\E\left[e^{-(r-\mu_1)\tau}\max\{1,(\Pi g_n)(0)\}\right]=\sup_{\tau\in \cT}\E\left[e^{-(r-\mu_1)\tau}\right]=1.
\end{align}
Finally, letting $n\to\infty$ in the last equation we also deduce $u(0)=1$. Then, by the same argument as in \eqref{eq:Pi0} we get that
\[
(\Pi u)(0)=\lim_{z\to 0}(\Pi u)(z)<1,
\]
which concludes the proof.
\end{proof}

Next, we will characterize the geometry of the stopping set. We denote by
\begin{align}\label{eq:Cu}
\cC:=\{z\in\R_+\,:\, u(z)>\max[1,(\Pi u)(z)]\}
\end{align}
the continuation set of problem \eqref{eq:u} and by
\begin{align}\label{eq:Su}
\cS:=\{z\in\R_+\,:\, u(z)=\max[1,(\Pi u)(z)]\}
\end{align}
its stopping set.

\begin{theorem}\label{thm:final}
If $F(0)=1$ we have $\cC=\varnothing$ and $u(z)=1+z$. If instead $F(0)<1$, then there exist two points $0<a_*<z_0<b_*<\infty$ such that $\cC=(a_*,b_*)$ with $z_0>0$ as in \eqref{eq:z_0}.
\end{theorem}
\begin{proof}
Let us first consider $F(0)=1$. From Lemma \ref{lem:z_0} we know that $z_0=0$. Therefore $\max\{1,(\Pi u)(z)\}=(\Pi u)(z)$ and the problem reduces to
\begin{align}
u(z)=\sup_{\tau\in\cT}\E_z\left[e^{-(r-\mu_1)\tau}(1+Z_\tau)\right]
\end{align}
by the same arguments as in Corollary \ref{cor:u}. Since the process $t\mapsto e^{-(r-\mu_1) t}(1+Z_t)$ is a supermartingale (Remark \ref{rem:properties}) then immediate stopping is optimal. That is, $\P_z(\hat \tau =0)=1$~(with $\hat \tau$ as in Corollary \ref{cor:OS-u}) and $u(z)=1+z$.

We now consider the case $F(0)<1$. This part of the proof is divided into three steps.

{\bf Step 1}. ({\em $z_0\in \cC$}.) Here we show that $z_0$ as in \eqref{eq:z_0} is always contained in the continuation region.  To simplify the notation we set $G(z):=\max\{1,(\Pi u)(z)\}$. Since the mapping $z\mapsto G(z)$ is convex (Proposition \ref{prop:u}), we can apply It\^o-Tanaka-Meyer formula (see \citet[Thm.~70, Ch.~IV]{protter2005stochastic}) to rewrite the stopping problem in the form
\begin{align}\label{eq:meyer}
u(z)=G(z)\!+\!\sup_{\tau\in \cT}\E\bigg[&\int_0^\tau \!\!e^{-(r-\mu_1)t}\left(G'(Z^z_s-)Z^z_s(\mu_2\!-\!\mu_1)\!-\!(r\!-\!\mu_1)G(Z^z_s)\right)\ud s\\
&+\!\tfrac{1}{2}\int_0^\infty\!\! L^a_\tau(Z^z) G''(\ud a)\bigg]
\end{align}
where $(L^a_t)_{t\ge 0}$ is the local time of the process $Z$ at a point $a\in\R_+$, $G''$ is understood as a non-negative measure, the left-derivative $G'(z-)$ is well defined by convexity and the martingale term can be removed thanks to standard localisation arguments.

Since $(\Pi u)$ is non-decreasing and convex with $(\Pi u)(0)<1$ (see proof of Lemma \ref{lem:z_0}), it is clear that $G''(\{z_0\})=(\Pi u)'(z_0+)>0$, where the existence of the right limit of the derivative of $(\Pi u)$ is guaranteed by convexity. Then, using \eqref{eq:meyer} with $z=z_0$ and $\tau=t\wedge\rho_\eps$ gives
\begin{align}
&u(z_0)-G(z_0)\\
&\ge \E\left[\int_0^{t\wedge\rho_\eps} \!\!e^{-(r-\mu_1)t}\left(G'(Z^{z_0}_s)Z^{z_0}_s(\mu_2\!-\!\mu_1)\!-\!(r\!-\!\mu_1)G(Z^{z_0}_s)\right)\ud s\!+\!\tfrac{1}{2}(\Pi u)'(z_0+)L^{z_0}_{t\wedge\rho_\eps}(Z^{z_0}) \right]
\end{align}
with $\rho_\eps:=\inf\{t\ge0:Z^{z_0}_t\notin(z_0\!-\!\eps,z_0\!+\!\eps)\}$, for $\eps>0$ given and fixed. Since $G$ and $G'$ are bounded on $[z_0\!-\!\eps,z_0\!+\!\eps]$ (recall that $G'$ is of bounded variation), we can find a constant $c_\eps>0$, independent of $t>0$, such that
\begin{align}
&u(z_0)-G(z_0)\ge -c_\eps\E\left[t\wedge\rho_\eps\right]\!+\!\tfrac{1}{2}(\Pi u)'(z_0+)\E\left[L^{z_0}_{t\wedge\rho_\eps}(Z^{z_0}) \right].
\end{align}
For small $t$ one has $\E\left[L^{z_0}_{t\wedge\rho_\eps}(Z^{z_0}) \right]\sim t^{1/2+\delta}$, for arbitrarily small $\delta>0$, and $\E\left[t\wedge\rho_\eps\right]\sim t$ (see, e.g., eqs.~(34) and (35) in \cite{de2017optimal}), hence the positive term dominates and $u(z_0)>G(z_0)$ as claimed.
\vspace{+5pt}

{\bf Step 2.} ({\em existence of $a_*$}.) Since $z\mapsto u(z)-1$ is non-decreasing and $u(z)\ge 1$, it is clear that if $z_1\in(0,z_0)$ belongs to $\cS$ then $[0,z_1]\subseteq \cS$. It remains to prove that it is possible to find one such $z_1$ strictly above zero, that is $a_*:=\sup\{z\in(0,z_0):z\in \cS\}>0$. We proceed by contradiction. Assume that $a_*=0$; then by the martingale property of $u$ (Corollary \ref{cor:OS-u}) we have that for any $z\in(0,z_0)$, setting $\rho_0=\rho_0(z)=\inf\{t\ge0:Z^z_t=z_0\}$, it holds
\begin{align}\label{eq:step2}
u(z)=&\E\left[e^{-(r-\mu_1)\rho_0}u(Z^z_{\rho_0})\right]=\E\left[e^{-(r-\mu_1)\rho_0}u(Z^z_{\rho_0})\I_{\{\rho_0<\infty\}}\right]\\
=&u(z_0)\E\left[e^{-(r-\mu_1)\rho_0(z)}\I_{\{\rho_0(z)<\infty\}}\right]=u(z_0)\left(\frac{z}{z_0}\right)^{q_1},
\end{align}
where the second equality holds because $u$ is bounded on compacts, the final expression is a representation of the Laplace transform of hitting times (see, e.g., \citet{borodin2012handbook}) and $q_1$ is the unique positive root of
\begin{align}\label{laplace}
\tfrac{1}{2}(\beta_1^2+\beta^2_2)q(q-1)+(\mu_2-\mu_1)q-(r-\mu_1)=0.
\end{align}

Letting $z\to 0$ in \eqref{eq:step2} we reach a contradiction because $u(z)\ge 1$ for all $z\in\R_+$. Hence we obtain the existence of $a_*>0$.
\vspace{+5pt}

{\bf Step 3.} ({\em existence of $b_*$}.) First we show that $z_2\!\in\! \cS\cap (z_0,\infty)$ implies $z_3\!\in\! \cS$ for all $z_3>z_2$.
Pick $z_2>z_0$ and assume $z_2\in \cS$. We again proceed by contradiction. Assume there is $z_3>z_2$ such that $z_3\in \cC$. By \eqref{eq:tauhat}, we let $\hat \tau (z_3)$ be the optimal stopping time for the problem starting at $z_3$. Then,
\[
0<\hat\tau(z_3)\le\inf\{t\ge 0: Z^{z_3}_t\le z_2\},\quad\text{$\P$-a.s.},
\]
and, since $z_2>z_0$, we have $\max\{1,(\Pi u)(Z^{z_3}_{\hat \tau})\}=(\Pi u)(Z^{z_3}_{\hat \tau})$. Then, using the supermartingale property \eqref{eq:superm} we have
\begin{align}\label{eq:b1}
u(z_3)=&\E\left[e^{-(r-\mu_1)\hat{\tau}}(\Pi u)(Z^{z_3}_{\hat \tau})\right]\le (\Pi u)(z_3).
\end{align}
The final inequality implies $z_3\in \cS$ as claimed.

Setting $b_*:=\inf\{z>z_0: z\in \cS\}$ it remains to show that $b_*<\infty$. We argue again by contradiction and assume that $(z_0,\infty)\subset \cC$. Then, taking an arbitrary $z>z_0$ and letting $\rho_0=\rho_0(z)$ be the hitting time to $z_0$ as in step 2, by the martingale property of the value function $u$ we obtain
\begin{align}\label{eq:step3}
u(z)=&\E\left[e^{-(r-\mu_1)\rho_0}u(Z^z_{\rho_0})\right]=\E\left[e^{-(r-\mu_1)\rho_0}u(Z^z_{\rho_0})\I_{\{\rho_0<\infty\}}\right]\\
=&u(z_0)\E\left[e^{-(r-\mu_1)\rho_0(z)}\I_{\{\rho_0(z)<\infty\}}\right]=u(z_0)\left(\frac{z}{z_0}\right)^{q_2}
\end{align}
where the second equality holds because of \eqref{eq:limsupu} and the final expression is a representation of the Laplace transform of hitting times, with $q_2$ being the unique negative root of \eqref{laplace}.

Letting $z\to\infty$ the right-hand side of \eqref{eq:step3} tends to zero, hence contradicting $u(z)\ge 1$ for all $z\in\R_+$. Therefore, existence of $b_*<\infty$ follows.
\end{proof}

The shape of the stopping region in the one dimensional problem \eqref{eq:u} translates into that of the original problem \eqref{eq:value2}, as stated in the corollary below.
\begin{corollary}\label{cor:final}
If $F(0)=1$, then the couple $(\tau_*,\alpha^*)=(0,1)$ is optimal in \eqref{eq:value2}. If instead $F(0)<1$, then an optimal couple $(\tau_*,\alpha^*)\in\D$ for problem \eqref{eq:value2} is given by
\begin{align}
\tau_*=\inf\{t\ge 0: K_t\notin ( S_t\cdot a_*,  S_t\cdot b_*)\}\quad\text{and}\quad\alpha^*=\I_{\{K_{\tau_*}\ge  S_{\tau_*}\cdot b_*\}}.
\end{align}
\end{corollary}

\begin{figure}
  \centering
  \includegraphics[width=10cm]{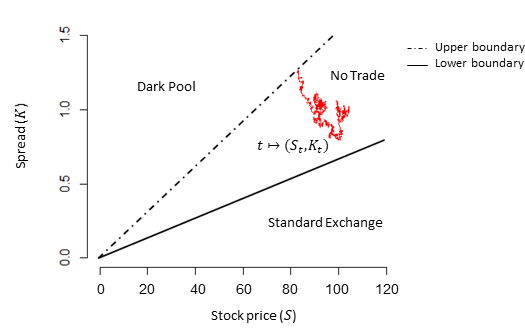}
  \caption{An illustration of continuation and stopping region. The continuation set is a wedge that separates two disconnected portions of the stopping region. These portions correspond to the choice of trading in the standard stock exchange (below the solid line) or in the dark pool (above the dashed line). The trajectory of the two dimensional GBM is simulated.}\label{fig:regions}
\end{figure}

The results of Theorem \ref{thm:final} and Corollary \ref{cor:final} have the following interpretation. First of all, we notice that holding the asset is penalised by effect of discounting (since $r>\max\{\mu_1, \mu_2\}$). Then the trader is unwilling to delay the sale for too long, irrespectively of how low/high the stock price is. Second, our model suggests that what matters in the trader's decision is the ratio between the spread and the stock price.
If the ratio between $K$ and  $S$ is very low (below $a_*$), the trader sells the stock in the standard exchange; indeed, in this case there is no additional benefit in attempting a sale in the dark pool, where the risk of a failed transaction is not compensated by a sufficiently large spread. If instead the ratio between the spread and the bid price is  large (above $b^*$), the trader is willing to take on the additional risk and attempts a sale in the dark pool, see Figure \ref{fig:regions}.

Finally, we comment on the fact that $z_0\in\cC$ (see \eqref{eq:z_0}). When the spread-price ratio equals $z_0$ the trader is faced with an extremely uncertain market condition.  Indeed, by definition $z_0$ is such that $s=(\Lambda v)(s,z_0 s)$. That is, at $z_0$ the payoff from a sale in the standard exchange market is equal to the expected one from a sale in the dark pool. Hence, it is natural for the trader to wait a little longer and see how the market is going to behave.

\subsection{Free boundary formulation}\label{sec:fbp}

Due to continuity of $u$ and thanks to standard optimal stopping theory we know that $u$ is in fact $C^2$ in the continuation set $\cC$ and it satisfies
\begin{align}\label{freeb0}
\big(\cL_Z-(r-\mu_1)\big)u(z)=0,\qquad\text{for $z\in\cC$},
\end{align}
where $\cL_Z$ is the generator of $Z$, that is
\[
(\cL_Z f)(z)=\frac{\beta_1^2+\beta^2_2}{2}z^2f''(z)+(\mu_2-\mu_1)f'(z)\quad\text{for all $f\in C^2(\R_+)$}.
\]

Now, notice that $u\in C^1(\R_+)$ implies that $\Pi u\in C^1(\R_+)$. The explicit dependence of the solution to equation \eqref{eq:Z3} on its initial point and an application of dominated convergence theorem allows us to write
\begin{align}\label{Pi'}
(\Pi u)'(z)=\int_0^\infty e^{-(r-\mu_1)t}\E\left[pZ^1_t+(1-p)u'(Z^z_t)Z^1_t\right]F(\ud t).
\end{align}

Using Proposition \ref{prop:C1} and Theorem \ref{thm:final} we can state the next result, which is the formal statement in this setting of the variational inequality introduced in Remark \ref{rem:VI}.

\begin{proposition}\label{prop:freeb}
Assume $F(0)<1$. Then  $(u,a_*,b_*)$ is the unique triple that solves the following problem:

Find $(\hat u, \hat a, \hat b)$ such that:
\begin{itemize}
\item[(i)] $\hat u\in C^1(\R_+)\cap C^2([\hat a,\hat b])$ and $\hat u$ is super-harmonic, i.e.,
\[
\E\left[e^{-(r-\mu_1)\tau}\hat u(Z^z_\tau)\right]\le \hat u(z),\qquad\text{for all $\tau\in\cT$};
\]
\item[(ii)] $\hat u\ge \max\{1,(\Pi \hat u)\}$ on $\R_+$, with $\hat u(z)>\max\{1,(\Pi \hat u)(z)\}$ iff $z\in(\hat a,\hat b)$;
\item[(iii)] the conditions below hold
\end{itemize}
\begin{align}
\label{freeb1}&\big(\cL_Z-(r-\mu_1)\big)\hat u(z)=0,\qquad\text{for $z\in(\hat a,\hat b)$,}\\
\label{freeb2}&\hat u(\hat a)=1,\quad \hat u(\hat b)=(\Pi\hat u)(\hat b),\quad\text{(continuous fit)},\\
\label{freeb3}&\hat u'(\hat a)=0,\quad \hat u'(\hat b)=(\Pi \hat u)'(\hat b),\quad\text{(smooth fit)}.
\end{align}
\end{proposition}
\begin{proof}
The fact that $u$ dominates $\max\{1,(\Pi u)\}$ and it is super-harmonic is given by Corollary \ref{cor:OS-u}, whilst $u>\max\{1,(\Pi u)\}$ on $(a_*,b_*)$ follows by definition of $a_*$ and $b_*$ in Theorem \ref{thm:final}. Thanks to Proposition \ref{prop:C1} and \eqref{freeb0}, we have that \eqref{freeb1}, \eqref{freeb2} and \eqref{freeb3} hold. Continuity of $u''$ on $[a_*,b_*]$ can be derived directly from \eqref{freeb1} by taking limits as $z\to \{a_*,b_*\}$ and noticing that the terms involving $u$ and $u'$ are continuous on $\R_+$.

As for uniqueness, by the super-harmonic property, we have that any $\hat u$ solving the problem (i)--(iii) also satisfies
\[
\hat u(z)\ge \E\left[e^{-(r-\mu_1)\tau}\hat u(Z^z_\tau)\right]\ge \E\left[e^{-(r-\mu_1)\tau}\max\{1,(\Pi \hat u)(Z^z_\tau)\right], \quad \text{for any $\tau\in\cT$}.
\]
Furthermore, using \eqref{freeb1} and the stopping time $\hat \tau_{a,b}:=\inf\{t\ge 0:Z_t\notin(\hat a,\hat b)\}$, we also obtain
\[
\hat u(z)= \E\left[e^{-(r-\mu_1)\tau}\max\{1,(\Pi\hat u)(Z^z_{\hat \tau_{a,b}})\right],
\]
hence
\[
\hat u(z)=\sup_{\tau\in \cT}\E\left[e^{-(r-\mu_1)\tau}\max\{1,(\Pi \hat u)(Z^z_\tau)\right].
\]

Uniqueness of the fixed point (see Theorem \ref{thm:prob1},  Corollary \ref{cor:u}) implies that $\hat u= u$ and therefore $(\hat a,\hat b)=(a_*,b_*)$.
\end{proof}

The free boundary formulation of the stopping problem may prove useful to compute (at least numerically) the values of $a_*$ and $b_*$, which can be used to find $u$. It should be noticed, however, that a direct solution of the free boundary problem, as commonly performed in one-dimensional optimal stopping problems, is far from being trivial in this case, because the boundary conditions at $b_*$ involve the value function itself in a non-local way.

Alternatively, one may follow a recursive scheme, based on the fixed point argument of Corollary \ref{cor:u}, in order to calculate approximating optimal boundaries associated with the stopping problems with value $\tilde\Gamma g_n$. Although it is easy to show\footnote{Notice that if $g_n\ge g_{n-1}$ then $\Pi g_n\ge \Pi g_{n-1}$ and $g_{n+1}=\tilde\Gamma g_n\ge \tilde\Gamma g_{n-1}=g_n$. Then, if $z\le a^{n+1}_*$ (i.e.~$g_{n+1}(z)=1$) we also have $g_{n}(z)=1$, hence $z\le a^n_*$. This implies $a^{n+1}_*\le a^n_*$ as claimed. Convergence of $a^n_*$ to $a_*$ follows from (monotonic) convergence of $g_n$ to $u$ in $\cA^+_1$.} that the corresponding sequence of approximating lower boundaries $a^n_*$ is decreasing and $a^n_*\downarrow a_*$, it seems much harder to determine monotonicity of the sequence of the approximating upper boundaries $b_*^n$.

\begin{remark}
The infinite-time horizon and the dimension reduction (Proposition \ref{prop:dim}) allowed us to use methods from the general theory of one-dimensional linear diffusions for the study of the value function and of the optimal boundaries in our optimal trading problem. That also led to a fairly explicit formulation of the associated free boundary problem in the proposition above.

Extending the analysis to a finite-time horizon set-up is non-trivial. Indeed, the dimension reduction is still feasible but it leads to a time-inhomogeneous optimal stopping problem for the process $(t,Z)$. Then the optimal boundaries in the reduced problem should be determined as functions
of time and we may reasonably expect the existence of maps $t\mapsto a_*(t)$ and $t\mapsto b_*(t)$ in place of the constant boundaries  found above. A key difficulty in the analysis is to determine properties of such time-dependent boundaries (e.g., monotonicity, continuity, etc.), which in turn will affect the smoothness of the value function. Traditionally these questions can be approached either via probabilistic methods as in, e.g., \cite{PS}, or via variational methods as in, e.g., \cite{Fr88}. In both instances,
classical results available in the literature cannot be directly applied to our case, because the recursive structure of our problem leads to boundary conditions that are {\em non-local}. Indeed, they depend in a functional manner on the solution (value function) via the operator $\Pi$ (or $\Lambda$ before dimension reduction) as illustrated in equations \eqref{freeb1}--\eqref{freeb3} or \eqref{eq:VI}. It seems however interesting to research how/if the PDE methods for parabolic free boundary problems with one spatial dimension (see, e.g., \cite{friedman1975parabolic} and \cite{Cannon} among the classical references) can be adapted to deal with our situation.
We leave this questions for future work.
\end{remark}

\subsection{Comparisons with a non-recursive model with dark pool}\label{sec:comparison}
Here we show that the optimal trading strategy obtained above is quantitatively and/or qualitatively different to optimal strategies arising from models without recursion.

The first simple observation is that in the absence of a dark pool the stock selling problem, under our assumption $\mu\le r$, becomes trivial as $e^{-rt}S_t$ is at best a martingale. Hence
\[
\sup_{\tau}\E_s\big[e^{-r\tau}S_\tau\big]\le s,
\]
by optional sampling theorem for positive (super)martingales and the optimal stopping rule prescribes immediate sale of the asset. If instead $\mu>r$ the problem is ill-posed: the trader would postpone the sale indefinitely since the discounted stock price is a sub-martingale. In this respect the addition of the dark pool introduces structural differences in the trader's optimal strategy.

Secondly, we investigate the impact of a recursive structure on the solution of the optimization problem, by comparing the optimal strategy found in Corollary \ref{cor:final} with the optimal strategy in an analogous problem but without recursion. As before we assume that a trader can place an order either in the lit market or in the dark pool. However, if the trade in the dark pool is not successful the trader turns directly to the lit market and sells the stock at the current price at time $\tau+\vartheta$. Then, the optimisation problem reads as in \eqref{eq:value2} but replacing $v(S_{\tau+\vartheta},K_{\tau+\vartheta})$ therein with $\gamma S_{\tau+\vartheta}$ (i.e., the sale price in the lit market if the transaction in the dark pool falls through). For the value function we have
\begin{align}\label{eq:hatv}
\hat v(s,k)=\sup_{\tau\in\cT}\E_{s,k}\big[e^{-r\tau}\max\{\gamma S_\tau,Q(S_\tau,K_\tau)\}\big],
\end{align}
where
\begin{align}\label{eq:Q}
Q(s,k):=\int_0^\infty e^{-rt}\E_{s,k}\big[p(S_t+K_t)+(1-p)\gamma S_t\big]F(\ud t).
\end{align}
Again, we set $\gamma=1$ for simplicity and with no loss of generality. The function $\hat v$ can be thought of as the value of a real option that gives a trader a one-off access to the dark pool.

Since $v(s,k)\ge s$ by construction (see \eqref{eq:value2}), then $Q(s,k)\le (\Lambda v)(s,k)$ for $(s,k)\in\R^2_+$. Hence, $\hat v\le v$ as expected: a one-off access to the dark pool is less valuable that an indefinite access to it. In particular the premium $\delta(s,k):=(v-\hat v)(s,k)$ is the additional cost that a trader should/may be willing to pay in order to gain indefinite access to the dark pool (as a real option) for the sale of a single stock.

Thanks to the linear structure of the payoff, once again the problem can be reduced to a one-dimensional set-up by the same transformation performed in the proof of Proposition \ref{prop:dim}. That is, setting
\[
A(z):=Q(1,z)=\int_0^\infty e^{-(r-\mu_1)t}\E_{z}\big[p(1+Z_t)+(1-p) \big]F(\ud t)
\]
and $\hat u(z)=\hat v(1,z)$ we have
\[
\hat u(z)=\sup_{\tau\in\cT}\E_z\big[e^{-(r-\mu_1)\tau}\max\{1,A(Z_\tau)\}\big].
\]

We focus on the more interesting case where $F(0)<1$. By repeating analogous calculations as in Section \ref{sec:boundaries} we obtain similar conclusions (due to the explicit form of the function $A$ computations are easier in this case). The function $A$ is increasing, so we can set
\[
\hat z_0=\inf\{z\in\R_+: A(z)>1\}
\]
and clearly $\hat z_0\ge z_0$ since $A(z)\le \Pi u(z)$ for $z\in\R_+$. Letting
\[
\hat \cC:=\big\{z\in\R_+: \hat u(z)>\max\{1,A(z)\}\big\}
\]
we can identify two optimal boundaries $0<\hat a<\hat z_0<\hat b <\infty$ so that $\hat \cC=(\hat a,\hat b)$. Then, as in Corollary \ref{cor:final}, the optimal pair $(\hat \tau,\hat \alpha)$ reads
\[
\hat \tau:=\inf\{t\ge 0: K_t\notin(S_t\cdot \hat a,S_t\cdot\hat b)\}\quad\text{and}\quad\hat \alpha:=\mathds{1}_{\{K_{\hat \tau}\ge S_{\hat \tau}\cdot\hat b\}}.
\]

Some interesting economic conclusions can be drawn by comparing the two sets $\cC$ and $\hat \cC$.
First of all, since $\hat u(z)\le u(z)$, then
\[
\{z\in\R_+: u(z)=1\}\subset \{z\in\R_+: \hat u(z)=1\},
\]
which implies
\[
\hat a\ge a_*.
\]
This means that the trader with indefinite access to the dark pool (i.e., whose objective is given in \eqref{eq:value2}) will delay the decision to sell the stock in the lit market compared to the trader with a one-off access to the dark pool (i.e., with objective as in \eqref{eq:hatv}). This is in line with the intuition that, since the dark pool is more attractive than the lit market and the `recursive' trader can attempt repeatedly a sale in the dark pool, the stock price $S$ must be sufficiently high (relative to the spread $K$) to convince a `recursive' trader to sell in the lit market. In contrast, the trader with a one-off chance to trade in the dark pool would be less inclined to delay a sale in the lit market when faced with a high stock price $S$ (relative to the spread), as the financial incentive offered by the dark pool is less pronounced.

By the same financial intuition we would also expect that the trader with unlimited access to the dark pool should attempt a sale in the dark pool earlier than the trader with a single opportunity. Due to the lack of explicit formulae for the solution of both problem \eqref{eq:value2} and \eqref{eq:hatv}, this conjecture is difficult to check in general. Nonetheless, we now show it is confirmed in the case $r=\mu_1$.

Recall that when $r=\mu_1$ we can still prove Theorem \ref{thm:prob1} but in the space $\cA'_2$ as explained in Remark \ref{rem:spaceAprime} (details are provided in the Appendix \ref{sec:appGBM}). In particular, the process $t\mapsto e^{-rt}(1+S_t+K_t)$ is a supermartingale if $(\mu_1-r) s+(\mu_2-r) k - r\le 0$. For $\mu_1=r$ the latter is verified for any $\mu_2\le r$ and we will assume $\mu_2<r$ to also guarantee that the analogue of (ii) and (iii) in Assumption \ref{ass:conv} hold (see Appendix \ref{sec:appGBM}).  Now, recalling the optimal stopping time $\tau_*$ for the recursive problem, we have
\begin{align}\label{eq:v-hatv}
v(s,k)-\hat v(s,k)\le \E_{s,k}\big[e^{-r\tau_*}\big(\Lambda v(S_{\tau_*},K_{\tau_*})-Q(S_{\tau_*},K_{\tau_*})\big)\big],
\end{align}
where the inequality holds because $\tau_*$ is sub-optimal for $\hat v(s,k)$ and $\max\{1,\Lambda v\}-\max\{1,Q\}\le \Lambda v-Q$, since $\Lambda v\ge Q$. Recalling the expressions for $\Lambda v$ and $Q$ in \eqref{eq:Lambdaf} and \eqref{eq:Q}, respectively, we have
\begin{align*}
&\E_{s,k}\Big[e^{-r\tau_*}\Big(\Lambda v(S_{\tau_*},K_{\tau_*})-Q(S_{\tau_*},K_{\tau_*})\Big)\Big]\\
&=(1-p)\E_{s,k}\Big[e^{-r\tau_*}\E_{S_{\tau_*},K_{\tau_*}}\Big[\int_0^\infty e^{-rt}\big(v(S_t,K_t)-S_t\big)F(\ud t)\Big]\Big]\\
&=(1-p)\int_0^\infty \E_{s,k}\Big[ e^{-r(\tau_*+t)}\big(v(S_{\tau_*+t},K_{\tau_*+t})-S_{\tau_*+t}\big)\Big]F(\ud t),
\end{align*}
by the strong Markov property and Fubini's theorem. Since $v\ge s$ we can combine Fatou's lemma and the optional sampling theorem to get
\begin{align*}
&\int_0^\infty \E_{s,k}\Big[ e^{-r(\tau_*+t)}\big(v(S_{\tau_*+t},K_{\tau_*+t})-S_{\tau_*+t}\big)\Big]F(\ud t)\\
&=\int_0^\infty \E_{s,k}\Big[\liminf_{n\to\infty} e^{-r((\tau_*\wedge n)+t)}\big(v(S_{\tau_*\wedge n+t},K_{\tau_*\wedge n+t})-S_{\tau_*\wedge n+t}\big)\Big]F(\ud t)\\
&\le\liminf_{n\to\infty}\int_0^\infty \E_{s,k}\Big[ e^{-r((\tau_*\wedge n + t)}\big(v(S_{\tau_*\wedge n + t},K_{\tau_*\wedge n + t})-S_{\tau_*\wedge n + t}\big)\Big]F(\ud t)\\
&\le\int_0^\infty \E_{s,k}\Big[ e^{-r t}\big(v(S_{t},K_{t})-S_{t}\big)\Big]F(\ud t)=\Lambda v(s,k)-Q(s,k),
\end{align*}
where in the final inequality we just used that $t\mapsto e^{-rt}v(S_{t},K_{t})$ is a supermartingale and $t\mapsto e^{-rt} S_t$ is a martingale on $[0,t+n]$ for each $n\ge 1$.

Combining the expressions above with \eqref{eq:v-hatv} we obtain
\[
v(s,k)-\hat v(s,k)\le \Lambda v(s,k)-Q(s,k),\quad \text{for $(s,k)\in\R^2_+$}.
\]
The latter implies
\[
\{(s,k)\in\R^2_+: \hat v(s,k)=Q(s,k)\}\subset \{(s,k)\in\R^2_+: v(s,k)=\Lambda v(s,k)\},
\]
which implies
\[
b_*\le \hat b
\]
as claimed.

In conclusion, we observe that the trader with unlimited access to the dark pool delays the sale in the lit market and anticipates the one in the dark pool, compared to a trader with a single access to the dark pool. From a geometric perspective we observe that the continuation wedge $\cC$ of the `recursive' problem is tilted clockwise compared to the one for the non-recursive problem, $\hat \cC$, although the aperture of the two wedges is not necessarily the same (see the illustration in Figure \ref{fig:nonrecursive}).

\begin{figure}
  \centering
  \includegraphics[width=10cm]{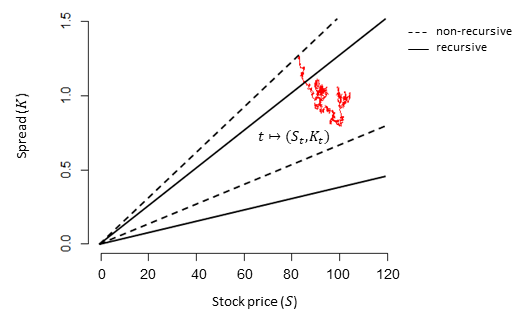}
  \caption{An illustration of continuation and stopping region for the recursive (solid line) and non-recursive (dashed line) trading problem. The continuation set corresponding to the recursive problem (the wedge between the solid lines) is tilted clockwise compared to the continuation region for the non-recursive problem (the wedge between the dashed lines). The aperture of the wedges needs not be the same.}
\label{fig:nonrecursive}
\end{figure}

\appendix

\section{The space $\cA_d$}
Here we show that $(\cA_d, \|\cdot\|_{\cA_d})$ is a Banach space. To start, we observe that $\|\cdot\|_{\cA_d}$ is a norm. To show completeness of the space we consider a Cauchy sequence $(f_n)_{n \in \mathbb{N}}\subset \cA_d$. Then for any given $\eps>0$ there is $N_{\eps}>0$ such that for every $n, m>N_{\eps}$ one has
$\|f_m-f_n\|_{\cA_d}<\eps$. By the definition of $\|\cdot\|_{\cA_d}$, we have that
\begin{equation}\label{eq:cauchy}
\frac{|f_m(x)-f_n(x)|}{(1+|x|_d^2)^{1/2}}<\eps,
\end{equation}
for every $x \in \R^d$ and every $n,m>N_{\eps}$. This implies that for each $x \in \R^d$ the sequence $(f_n(x))_{n\in\mathbb{N}}\subset \R$ is Cauchy. Therefore there exists a function $f:\R^d\to\R$ such that
\begin{equation}
f(x)=\lim_{m \to \infty} f_m(x),
\end{equation}
for all $x\in \R^d$. If we take the limit as $m \to \infty$ in \eqref{eq:cauchy} we get that
\begin{equation}
\frac{|f(x)-f_n(x)|}{(1+|x|_d^2)^{1/2}}<\eps,
\end{equation}
for every $x \in \R^d$ and every $n>N_{\eps}$. Hence $\|f-f_n\|_{\cA_d}<\eps$ and $\|f\|_{\cA_d}<\infty$.

To conclude that $f \in \cA_d$ we need to show that $f$ is continuous. Let $(f_n)_{n \in \mathbb{N}}\subset \cA_d$ be the  Cauchy sequence from the paragraph above with $f_n\to f$ as $n \to \infty$. Fix $x_0\in \R^d$ and take a sequence $(x_k)_{k \in \mathbb{N}}\subset\R^d$ such that $x_{k}\to x_0$ as $k \to \infty$. Without loss of generality we can assume that there is a compact $U\subset\R^d$ such that $x_0\in U$ and $(x_k)_{k \in \mathbb{N}}\subset U$.  Fix $\eps>0$, then for any $n\ge N_\eps$ we have
\begin{align}
|f(x_k)-f(x_0)|&\leq( |f(x_k)-f_n(x_k)|+|f_n(x_k)-f_n(x_0)|+|f_n(x_0)-f(x_0)|)\\
&\leq \|f_n-f\|_{\cA_d} (2+|x_k|_d^2+|x_0|_d^2)^{1/2}+|f_n(x_k)-f_n(x_0)|\\
&\le c_U\eps+|f_n(x_k)-f_n(x_0)|,
\end{align}
where $c_U:=[2\sup_{x\in U}(1+|x|^2_d)]^{1/2}$. Thanks to arbitrariness of $\varepsilon$, taking limits as  $k\to\infty$ proves continuity. 

\section{Existence of the value function in the case of $\mu_1=r$}\label{sec:appGBM}

Here we provide the details for the proof of Theorem \ref{thm:prob1} in the setting of Section \ref{sec:comparison}.
In particular we assume $\mu_1=r$ and $\mu_2<r$. Under this assumption some care is needed since $(e^{-rt}S_t)_{t\ge 0}$ is no longer a uniformly integrable process but it is a positive martingale
for $t\in[0,\infty)$, hence a (positive) supermartingale for $t\in[0,\infty]$ with
\[
e^{-r\tau}S_{\tau}\mathds 1_{\{\tau=\infty\}}:=\lim_{t\to\infty}e^{-rt}S_{t}=0, \quad\text{$\P_{s}$-a.s.\ for all $s\in\R_+$.}
\]
Therefore $\E_s[e^{-r\tau}S_\tau]\le s$ for any $\tau\in\cT$ by optional sampling (see \cite[Thm.3.22, Ch.1]{KS2}). Moreover, recalling Lemma \ref{lem:value},
it is sufficient to consider the problem
\begin{align}\label{eq:t-v}
v(s,k)=\sup_\tau\E_{s,k}\big[e^{-r\tau}\max\{S_\tau,(\Lambda v)(S_\tau,K_\tau)\}\big].
\end{align}
Lemma \ref{lem:homo} continues to hold with the same proof. Then, setting $z=k/s$ and $s u(z):=s v(1,z)=v(s,k)$ it is clear that $u\in\cA'_1$ iff $v\in\cA'_2$ (recall $\cA'_2$ as in \eqref{eq:A2b} and notice that $\cA'_1$ is the analogue for the one dimensional process $Z$ from \eqref{eq:Z3}). Repeating the same arguments of proof as in Proposition \ref{prop:dim} we conclude that $v\in\cA'_2$ and the problem in \eqref{eq:t-v} is well-posed iff $u\in\cA'_1$ and it solves
\[
u(z)=\sup_\tau\E_{z}\big[\max\{1,(\Pi u)(Z_\tau)\}\big].
\]
Here we have set
\[
\Pi f(z)=\int_0^\infty\E_z\big[p(1+Z_t)+(1-p)f(Z_t)\big]F(\ud t),\quad\text{for any $f\in\cA'_1$},
\]
since $\mu_1=r$, and the dynamics of $Z$ under $\P_z$ reads (cf.\ \eqref{eq:Z3})
\[
\ud Z_t=(\mu_2-r)Z_t\ud t+\sqrt{\beta^2_1+\beta^2_2}Z_t\ud \tilde{B}_t,\qquad Z_0=z.
\]
Notice that since $\mu_2<r$ the process $t\mapsto Z_t$ is a positive supermartingale for $t\in[0,\infty]$ with
\begin{align}\label{eq:limZ0}
\lim_{t\to\infty}Z_t=0,\quad \text{$\P_z$-a.s.\ for all $z\in[0,\infty)$.}
\end{align}
Using the same arguments as in Lemma \ref{lem:Lambda} we have $f\in\cA'_1\implies \Pi f\in\cA'_1$, thanks to the supermartingale property of $Z$. Then, by the linear growth of $f\in\cA'_1$ and the supermartingale property of $Z$ we can easily show that
\begin{align}\label{eq:subZ}
|(\tilde\Gamma f)(z)|\le c(1+z)
\end{align}
for some $c>0$ (with $\tilde\Gamma$ as in \eqref{eq:ug}).

Next we want to show that Lemma \ref{lem:OS} holds for $\tilde \Gamma f$. In order to apply the results from general optimal stopping theory as in the proof of Lemma \ref{lem:OS},  we need to check the analogue of condition \eqref{eq:usc}, i.e.,
\[
\E_z\Big[\sup_{t\ge 0}\max\{1,(\Pi f)(Z_t)\}\Big]<\infty.
\]
By linear growth of $\Pi f$, this boils down to verifying
\begin{align}\label{eq:supZ}
\E_z\Big[\sup_{t\ge 0}Z_t\Big]<\infty.
\end{align}

The latter holds because, setting $\beta:=\sqrt{\beta^2_1+\beta^2_2}$, $\kappa:=r+\tfrac{\beta^2}{2}-\mu_2>0$ and $Y_t:=-\kappa t+\beta \tilde B_t$ for notational convenience, we have
\begin{align}\label{eq:supZ2}
\E_z\Big[\sup_{t\ge 0}Z_t\Big]=&z\E\Big[\exp\Big(\sup_{t\ge 0}Y_t\Big)\Big]=z\,\tfrac{2\kappa}{\beta^2}\int_0^\infty e^{y-\tfrac{2\kappa}{\beta^2}y}\ud y=:z\, c_0<\infty,
\end{align}
where we used $\P(\sup_{t\ge 0}Y_t>y)=\exp (-\tfrac{2\kappa}{\beta^2}y)$ \cite[Ch.~1, Sec.~3.5.C]{KS2} and the integral is finite because $2\kappa>\beta^2$ thanks to $\mu_2<r$.
Combining \eqref{eq:limZ0}--\eqref{eq:supZ2}, we also have
\begin{align}\label{eq:lim-supZ}
\lim_{s\to\infty}\P_z\Big(\sup_{t\ge s}Z_t>\eps\Big)=&\lim_{s\to\infty}\E_z\Big[\P_{Z_s}\Big(\sup_{t\ge 0}Z_t>\eps\Big)\Big]\\
\le & \frac{1}{\eps}\lim_{s\to\infty}\E_z\Big[\E_{Z_s}\Big[\sup_{t\ge 0}Z_t\Big]\Big]\\
=&\frac{c_0}{\eps}\lim_{s\to\infty}\E_z\big[Z_s\big]=\frac{c_0}{\eps}\E_z\big[\lim_{s\to\infty}Z_s\big]=0,\quad\text{for all $\eps>0$},
\end{align}
where the inequality is Markov's inequality, the first equality is by the Markov property of $Z$, the second one by \eqref{eq:supZ2} and the final one by dominated convergence and \eqref{eq:supZ}.

The only remaining hurdle to prove Theorem \ref{thm:prob1} in this context is the continuity of $z\mapsto \tilde\Gamma f(z)$ for $f\in\cA'_1$. Indeed, while lower semi-continuity follows from Lemma \ref{lem:OS}, the proof of upper semi-continuity needs a different argument. If we simply repeat the estimates in the proof of Lemma \ref{lem:usc-G}, in the final equation in step 1 we can no longer let $S\to\infty$ since $r=0$. We follow a slightly different route taking advantage of the explicit nature of the dynamics.

Given a sequence $(z_n)_{n\ge 1}$ converging to $z$, as in \eqref{eq:usc1} we have
\begin{align*}
\tilde\Gamma f(z_n)-\tilde\Gamma f(z)\le& \E\big[\max\{1,(\Pi f)(Z^{z_n}_{\tau_n})\}-\max\{1,(\Pi f)(Z^{z}_{\tau_n})\}\big]\\
\le& \E\big[\big|(\Pi f)(Z^{z_n}_{\tau_n})-(\Pi f)(Z^{z}_{\tau_n})\big|\big].
\end{align*}
Fix $\eps>0$, then for any $\delta\in(0,1)$ there exists $s_{\eps,\delta}>0$ such that $\P\big(\sup_{t\ge s_{\eps,\delta}}Z^1_t>\eps\big)\le \delta$ thanks to \eqref{eq:lim-supZ}. Then,
\begin{align}\label{eq:cont-r0}
&\tilde\Gamma f(z_n)-\tilde\Gamma f(z)\\
&\le \E\big[\big|(\Pi f)(z_nZ^{1}_{\tau_n})-(\Pi f)(zZ^{1}_{\tau_n})\big|\mathds 1_{\{\tau_n\le s_{\eps,\delta}\}}\big]\notag\\
&\quad +\E\big[\big|(\Pi f)(z_nZ^{1}_{\tau_n})-(\Pi f)(zZ^{1}_{\tau_n})\big|\mathds 1_{\{\tau_n> s_{\eps,\delta}\}}\big]\notag\\
&\le \E\big[\sup_{0\le t\le s_{\eps,\delta}}\big|(\Pi f)(z_nZ^{1}_{t})-
(\Pi f)(zZ^{1}_{t})\big|\mathds 1_{\{\tau_n\le s_{\eps,\delta}\}}\big]\notag\\
&\quad +\E\big[\big|(\Pi f)(z_nZ^{1}_{\tau_n})-(\Pi f)(zZ^{1}_{\tau_n})\big|\mathds 1_{\{\tau_n> s_{\eps,\delta}\}}\big].\notag
\end{align}

Let us consider the first term in the final expression above. It is immediate that
\[
\E\big[\sup_n\big(\sup_{t\ge 0}Z^{z_n}_t\big)\big]\le \E\big[\sup_{t\ge 0}Z^{z}_t\big]\sup_n(z_n/z)<\infty
\]
and therefore, using the linear growth of $\Pi f$ we have
\[
\E\big[\sup_{n}\sup_{t\ge 0}\big|(\Pi f)(Z^{z_n}_{t})-(\Pi f)(Z^{z}_{t})\big|\big]<\infty.
\]
Then, by dominated convergence we have
\begin{align*}
&\limsup_{n\to\infty}\E\big[\sup_{0\le t\le s_{\eps,\delta}}\big|(\Pi f)(z_nZ^{1}_{t})-
(\Pi f)(zZ^{1}_{t})\big|\mathds 1_{\{\tau_n\le s_{\eps,\delta}\}}\big]\\
&\le\E\big[\limsup_{n\to\infty}\sup_{0\le t\le s_{\eps,\delta}}\big|(\Pi f)(z_nZ^{1}_{t})-
(\Pi f)(zZ^{1}_{t})\big|\big]=0
\end{align*}
where the final equality is due to the fact that $(t,y)\mapsto \big|(\Pi f)(yZ^{1}_{t})-
(\Pi f)(zZ^{1}_{t})\big|$ is continuous, hence uniformly continuous on compacts. This takes care of the first term in the final expression of \eqref{eq:cont-r0}.

For the other term we denote $A_{\eps,\delta}:=\{\omega:\sup_{t\ge s_{\eps,\delta}}Z^1_t(\omega)>\eps\}$ and by $A^c_{\eps,\delta}$ its complement. Then we split the expectation on the events
\[
\{\tau_n>s_{\eps,\delta}\}\cap A_{\eps,\delta}\quad\text{and}\quad \{\tau_n>s_{\eps,\delta}\}\cap A^c_{\eps,\delta}.
\]
On the first event we have
\begin{align*}
&\limsup_{n\to\infty}\E\big[\big|(\Pi f)(z_nZ^{1}_{\tau_n})-(\Pi f)(zZ^{1}_{\tau_n})\big|\mathds 1_{\{\tau_n> s_{\eps,\delta}\}\cap A_{\eps,\delta}}\big]\\
&\le \limsup_{n\to\infty} (z_n+z)\|\Pi f\|_{\cA'_1}\E\big[\sup_{t\ge 0} Z^{1}_{t} \mathds 1_{A_{\eps,\delta}}\big]\\
&\le 2z\|\Pi f\|_{\cA'_1}\E\big[\sup_{t\ge 0} Z^{1}_{t} \mathds 1_{A_{\eps,\delta}}\big].
\end{align*}
On the other event we have
\begin{align*}
&\limsup_{n\to\infty}\E\big[\big|(\Pi f)(z_nZ^{1}_{\tau_n})-(\Pi f)(zZ^{1}_{\tau_n})\big|\mathds 1_{\{\tau_n> s_{\eps,\delta}\}\cap A^c_{\eps,\delta}}\big]\\
&\le\limsup_{n\to\infty} \E\big[\sup_{0\le y\le \eps}\big|(\Pi f)(z_n y)-(\Pi f)(z y)\big|\big]\\
&\le \E\big[\limsup_{n\to\infty}\sup_{0\le y\le \eps}\big|(\Pi f)(z_n y)-(\Pi f)(z y)\big|\big]=0,
\end{align*}
by dominated convergence and uniform continuity of $(\zeta,y)\mapsto\big|(\Pi f)(\zeta y)-(\Pi f)(z y)\big|$ on compacts.

Combining the estimates above we get
\begin{align*}
&\limsup_{n\to\infty}\tilde\Gamma f(z_n)-\tilde\Gamma f(z)\le 2z\|\Pi f\|_{\cA'_1}\E\big[\sup_{t\ge 0} Z^{1}_{t} \mathds 1_{A_{\eps,\delta}}\big].
\end{align*}
Letting $\delta\to 0$ we have $s_{\eps,\delta}\to\infty$ and $\P(A_{\eps,\delta})\downarrow 0$. Then, by \eqref{eq:supZ} and monotone convergence we conclude that
$\limsup_{n\to\infty}\tilde\Gamma f(z_n)-\tilde\Gamma f(z)\le 0$ as needed.

Thanks to the continuity, we have $\tilde \Gamma f\in\cA'_1$ for all $f\in\cA'_1$ and repeating the arguments of proof of Theorem \ref{thm:prob1} the operator $\tilde \Gamma$ is also a contraction. Hence, all the results in Theorem \ref{thm:prob1} continue to hold.

\bibliographystyle{plainnat}
\bibliography{bibliography}

\begin{thebibliography}{33}
\providecommand{\natexlab}[1]{#1}
\providecommand{\url}[1]{\texttt{#1}}
\expandafter\ifx\csname urlstyle\endcsname\relax
  \providecommand{\doi}[1]{doi: #1}\else
  \providecommand{\doi}{doi: \begingroup \urlstyle{rm}\Url}\fi

\bibitem[Bar-Ilan and Strange(1996)]{bar1996investment}
A.~Bar-Ilan and W.~C. Strange.
\newblock Investment lags.
\newblock \emph{Am.\ Econ.\ Rev.}, 86\penalty0 (3):\penalty0 610--622, 1996.

\bibitem[Bar-Ilan and Sulem(1995)]{bar1995explicit}
A.~Bar-Ilan and A.~Sulem.
\newblock Explicit solution of inventory problems with delivery lags.
\newblock \emph{Math.\ Oper.\ Res.}, 20\penalty0 (3):\penalty0 709--720, 1995.

\bibitem[Bayraktar and Egami(2007)]{bayraktar2007effects}
E.~Bayraktar and M.~Egami.
\newblock The effects of implementation delay on decision-making under
  uncertainty.
\newblock \emph{Stoch.\ Process.\ Appl.}, 117\penalty0 (3):\penalty0 333--358,
  2007.

\bibitem[Boni et~al.(2013)Boni, Brown, and Leach]{boni2013dark}
L.~Boni, D.~C. Brown, and J.~C. Leach.
\newblock Dark pool exclusivity matters.
\newblock \emph{Available at SSRN 2055808}, 2013.

\bibitem[Borodin and Salminen(2012)]{borodin2012handbook}
A.N. Borodin and P.~Salminen.
\newblock \emph{Handbook of Brownian motion-facts and formulae}.
\newblock Birkh{\"a}user, 2012.

\bibitem[Buti et~al.(2017)Buti, Rindi, and Werner]{buti2017dark}
S.~Buti, B.~Rindi, and I.~M. Werner.
\newblock Dark pool trading strategies, market quality and welfare.
\newblock \emph{J.\ Financ.\ Econ.}, 124\penalty0 (2):\penalty0 244--265, 2017.

\bibitem[Cannon(1984)]{Cannon}
J.~R. Cannon.
\newblock \emph{The one-dimensional heat equation}, volume~23 of
  \emph{Encyclopedia of Mathematics and its Applications}.
\newblock Addison-Wesley Publishing Company, Advanced Book Program, Reading,
  MA, 1984.

\bibitem[Carmona and Touzi(2008)]{carmona2008optimal}
R.~Carmona and N.~Touzi.
\newblock Optimal multiple stopping and valuation of swing options.
\newblock \emph{Math.\ Finance}, 18\penalty0 (2):\penalty0 239--268, 2008.

\bibitem[Cartea et~al.(2015)Cartea, Jaimungal, and
  Penalva]{cartea2015algorithmic}
{\'A}.~Cartea, S.~Jaimungal, and J.~Penalva.
\newblock \emph{Algorithmic and high-frequency trading}.
\newblock Cambridge University Press, 2015.

\bibitem[Crisafi and Macrina(2016)]{crisafi2016simultaneous}
M.A. Crisafi and A.~Macrina.
\newblock Simultaneous trading in lit and dark pools.
\newblock \emph{Int.\ J.\ Theor.\ Appl.\ Finance}, 19\penalty0 (08):\penalty0
  1650055, 2016.

\bibitem[Dayanik and Karatzas(2003)]{dayanik2003optimal}
S.~Dayanik and I.~Karatzas.
\newblock On the optimal stopping problem for one-dimensional diffusions.
\newblock \emph{Stoch.\ Process.\ Appl.}, 107\penalty0 (2):\penalty0 173--212,
  2003.

\bibitem[De~Angelis and Kitapbayev(2017)]{de2017optimal}
T.~De~Angelis and Y.~Kitapbayev.
\newblock On the optimal exercise boundaries of swing put options.
\newblock \emph{Math.\ Oper.\ Res.}, 43\penalty0 (1):\penalty0 252--274, 2017.

\bibitem[Degryse et~al.(2009)Degryse, Van~Achter, and
  Wuyts]{degryse2009shedding}
H.~Degryse, M.~Van~Achter, and G.~Wuyts.
\newblock Shedding light on dark liquidity pools.
\newblock \emph{The Institutional Investor}, 2009\penalty0 (1):\penalty0
  147--155, 2009.

\bibitem[Dixit and Pindyck(1994)]{dixit1994}
A.K. Dixit and R.S. Pindyck.
\newblock \emph{Investment under uncertainty}.
\newblock Princeton university press, 1994.

\bibitem[Duffie and Epstein(1992)]{duffie1992stochastic}
D.~Duffie and L.G. Epstein.
\newblock Stochastic differential utility.
\newblock \emph{Econometrica}, pages 353--394, 1992.

\bibitem[Dynkin and Yushkevich(1969)]{dynkin1969markov}
E.B. Dynkin and A.A. Yushkevich.
\newblock \emph{Markov processes: Theorems and problems}.
\newblock Plenum, 1969.

\bibitem[Epstein and Zin(2013)]{epstein2013substitution}
L.G. Epstein and S.E. Zin.
\newblock Substitution, risk aversion and the temporal behavior of consumption
  and asset returns: A theoretical framework.
\newblock In \emph{Handbook of the Fundamentals of Financial Decision Making:
  Part I}, pages 207--239. World Scientific, 2013.

\bibitem[Friedman(1975)]{friedman1975parabolic}
A.~Friedman.
\newblock Parabolic variational inequalities in one space dimension and
  smoothness of the free boundary.
\newblock \emph{J.\ Funct.\ Anal.}, 18\penalty0 (2):\penalty0 151--176, 1975.

\bibitem[Friedman(1988)]{Fr88}
A.~Friedman.
\newblock \emph{Variational principles and free-boundary problems}.
\newblock Robert E. Krieger Publishing Co., Inc., Malabar, FL, second edition,
  1988.
\newblock ISBN 0-89464-263-4.

\bibitem[Ganchev et~al.(2010)Ganchev, Nevmyvaka, Kearns, and
  Vaughan]{ganchev2010censored}
K.~Ganchev, Y.~Nevmyvaka, M.~Kearns, and J.~W. Vaughan.
\newblock Censored exploration and the dark pool problem.
\newblock \emph{Commun.\ ACM}, 53\penalty0 (5):\penalty0 99--107, 2010.

\bibitem[Karatzas and Shreve(1998{\natexlab{a}})]{KS}
I.~Karatzas and S.E. Shreve.
\newblock \emph{Methods of mathematical finance}, volume~39.
\newblock Springer, 1998{\natexlab{a}}.

\bibitem[Karatzas and Shreve(1998{\natexlab{b}})]{KS2}
I.~Karatzas and S.E. Shreve.
\newblock \emph{Brownian Motion and Stochastic Calculus}.
\newblock Springer, 1998{\natexlab{b}}.

\bibitem[Kratz and Sch{\"o}neborn(2015)]{kratz2015portfolio}
P.~Kratz and T.~Sch{\"o}neborn.
\newblock Portfolio liquidation in dark pools in continuous time.
\newblock \emph{Math.\ Finance}, 25\penalty0 (3):\penalty0 496--544, 2015.

\bibitem[Kratz and Sch{\"o}neborn(2018)]{kratz2018optimal}
P.~Kratz and T.~Sch{\"o}neborn.
\newblock Optimal liquidation and adverse selection in dark pools.
\newblock \emph{Math.\ Finance}, 28\penalty0 (1):\penalty0 177--210, 2018.

\bibitem[Krylov(2008)]{Krylov}
N.V. Krylov.
\newblock \emph{Controlled diffusion processes}, volume~14.
\newblock Springer Science \& Business Media, 2008.

\bibitem[Lempa(2012)]{lempa2012optimal}
J.~Lempa.
\newblock Optimal stopping with random exercise lag.
\newblock \emph{Math.\ Methods Oper.\ Res.}, 75\penalty0 (3):\penalty0
  273--286, 2012.

\bibitem[Mittal(2008)]{mittal2008you}
H.~Mittal.
\newblock Are you playing in a toxic dark pool?: A guide to preventing
  information leakage.
\newblock \emph{J.\ Trading}, 3\penalty0 (3):\penalty0 20--33, 2008.

\bibitem[{\O}ksendal(2005)]{oksendal2005optimal}
B.~{\O}ksendal.
\newblock Optimal stopping with delayed information.
\newblock \emph{Stoch.\ Dyn.}, 5\penalty0 (02):\penalty0 271--280, 2005.

\bibitem[Peskir and Shiryaev(2006)]{PS}
G.~Peskir and A.N. Shiryaev.
\newblock \emph{Optimal stopping and free-boundary problems}.
\newblock Springer, 2006.

\bibitem[Protter(2005)]{protter2005stochastic}
P.E. Protter.
\newblock \emph{Stochastic integration and differential equations}.
\newblock Springer, 2005.

\bibitem[Shiryaev(1988)]{shiryaev1988probability}
A.N. Shiryaev.
\newblock \emph{Probability (Second edition)}.
\newblock Springer-Verlag, Berlin Heidelberg Germany, 1988.

\bibitem[Shiryaev(2007)]{shiryaev2007optimal}
A.N. Shiryaev.
\newblock \emph{Optimal stopping rules}, volume~8.
\newblock Springer Science \& Business Media, 2007.

\bibitem[Zhu(2014)]{zhu2014dark}
H.~Zhu.
\newblock Do dark pools harm price discovery?
\newblock \emph{Rev.\ Financ.\ Stud.}, 27\penalty0 (3):\penalty0 747--789,
  2014.

\end{thebibliography}

\end{document}